\documentclass[reqno,12pt]{amsart}
\usepackage{geometry}
\usepackage{cite}
\usepackage{epsfig}
\usepackage{amsmath}
\usepackage{amsfonts}
\usepackage{mathrsfs}
\usepackage{subfigure}
\usepackage[colorlinks,
            linkcolor=blue,
            anchorcolor=blue,
            citecolor=blue
            ]{hyperref}

\numberwithin{equation}{section}
\allowdisplaybreaks[4]

\begin{document}
\theoremstyle{plain}
\newtheorem{theorem}{Theorem}[section]
\newtheorem{lemma}[theorem]{Lemma}
\newtheorem{corollary}[theorem]{Corollary}
\newtheorem{proposition}[theorem]{Proposition}

\theoremstyle{definition}
\newtheorem{definition}[theorem]{Definition}
\newtheorem{example}[theorem]{Example}

\theoremstyle{remark}
\newtheorem{remark}{Remark}
\newtheorem{notation}{Notation}

\newcommand{\IN}{\mathbb{N}}
\newcommand{\IR}{\mathbb{R}}
\newcommand{\la}{\lambda}
\newcommand{\s}{\sum_{n=0}^{\infty}}
\newcommand{\sk}{\sum_{k=0}^n}
\newcommand{\skk}{\sum_{k=1}^n}
\newcommand{\D}{\mathcal{D}}
\newcommand{\T}{\mathbb{T}}

\numberwithin{equation}{section}
\allowdisplaybreaks[4]

\def\square{\hfill${\vcenter{\vbox{\hrule height.4pt \hbox{\vrule width.4pt
height7pt \kern7pt \vrule width.4pt} \hrule height.4pt}}}$}

\title[]
{\small Notes on $q$-partial differential equations for $q$-Laguerre polynomials and little $q$-Jacobi polynomials}


\author{Qi Bao$^{1,\ast}$}

\address{Qi Bao$^{1}$, $^{1}$School of Mathematical Sciences, East China Normal University, Shanghai, 200241, China}

\email{52205500010@stu.ecnu.edu.cn}

\author{DunKun Yang$^{2}$}

\address{DunKun Yang$^{2}$, $^{2}$School of Mathematical Sciences, East China Normal University, Shanghai, 200241, China}

\email{52265500003@stu.ecnu.edu.cn}



\thanks{$^{\ast}$Corresponding author. Email address: 52205500010@stu.ecnu.edu.cn}

\begin{abstract}
We define two common $q$-orthogonal polynomials: homogeneous $q$-Laguerre polynomials and homogeneous little $q$-Jacobi polynomials. They can be viewed separately as solutions to two $q$-partial differential equations. Then, we proved that if an analytic function satisfies a certain system of $q$-partial differential equations, if and only if it can be expanded in terms of homogeneous $q$-Laguerre polynomials or homogeneous little $q$-Jacobi polynomials. As applications, we obtain generalizations of the Ramanujan $q$-beta integrals and Andrews-Askey integrals. Additionally, we present an operator representation of $q$-Laguerre polynomials that facilitates the computation of identities involving $q$-Laguerre polynomials.
\end{abstract}











\maketitle

\section{Introduction}

The presence of orthogonal polynomials is ubiquitous in various problems encountered in classical mathematical physics. For instance, the Hermite polynomials manifest in the quantum mechanical treatment of harmonic oscillators, while the Laguerre polynomials arise in the propagation of electromagnetic waves. However, the study of $q$-orthogonal polynomials is also a crucial study topic and can be found in relevant literature \cite{Al1965,Anne1994,Heine1878,Lebedev}.

Throughout the paper, it is supposed that $0<|q|<1$, which ensures that all the sums and products appear in the paper converge, and denote by $\IN$ ($\mathbb{C}$) the set of positive integers (complex numbers, respectively). The $q$-shifted factorials are defined as
\begin{align*}
(a;q)_0=1, \,\, (a;q)_n=\prod_{k=0}^{n-1}(1-aq^k), \,\,
(a;q)_{\infty}=\prod_{k=0}^{\infty}(1-aq^k).
\end{align*}
For any function $f(x)$ of one variable, the $q$-derivative of $f(x)$ with respect to $x$ is defined by
\begin{align*}
\D_q \{f(x)\}=\frac{f(x)-f(qx)}x.
\end{align*}
According to the above definition, it is not difficult to verify
\begin{align}
\D_q \{f(x)g(x)\}&= \D_q \{f(x)\}g(x)+f(qx) \D_q \{g(x)\}
 \quad \label{Dfg}
\end{align}
and the Leibniz rule for the product of two functions
\begin{align}
\D_q^n \{f(x)g(x)\}&=\sk \begin{bmatrix} n \\k \end{bmatrix}q^{k(k-n)}
\D_q^k \{f(x)\} \D_q^{n-k} \{g(q^k x)\} \label{Dfg1},
\end{align}
where
\begin{align}\label{binomial}
\begin{bmatrix} n \\k \end{bmatrix}
=\frac{ (q;q)_n }{ (q;q)_k (q;q)_{n-k} },
\,\, 0\leq k \leq n, \,\,  n\in\IN
\end{align}
is the Gaussian binomial coefficients, also see \cite{GR}. For any real number $r$, the $q$-shift operator $\eta_{x_i}^r$ is defined by
\begin{align*}
\eta_{x_i}^r \{ f(x_1,\cdots,x_n) \}=f(x_1,\cdots,x_{i-1},q^r x_i,x_{i+1},\cdots,x_n).
\end{align*}
Generalizing Heine's series, or basic hypergeometric series ${}_r\phi_s$ is defined by
\begin{align}\label{rs}
{}_r\phi_s \left( \begin{gathered}
a_1, a_2, \cdots , a_r \\ b_1, b_2, \cdots , b_s \end{gathered}
;\,q, z \right)=\s \frac{ (a_1;q)_n  \cdots (a_r;q)_n  }
{ (b_1;q)_n  \cdots (b_s;q)_n (q;q)_n}
\left[ (-1)^n q^{ \binom n2 } \right]^{1+s-r} z^n.
\end{align}
Here and in what follows, $\binom nk$ represents the standard combination symbol. The series ${}_r\phi_s$ terminates if one of the numerator parameters is of the form $q^{-n}$, $n\in \IN \cup \{0\}$ and $q\neq0$. If $0<|q|<1$, the series ${}_r\phi_s$ converges absolutely
for all $x$ if $r\leq s$ and for $|x|<1$ if $r=s+1$. The famous $q$-binomial theorem
\begin{align}\label{binomialTh}
{}_1\phi_0\left( \begin{gathered} a \\ - \end{gathered} ;\,q, z \right)
=\s \frac{(a;q)_n}{(q;q)_n}z^n
= \frac{(az;q)_{\infty}}{(z;q)_{\infty}}, \,|z|<1,
\end{align}
is a $q$-analogue of Newton's binomial series. This theorem can also derive the following two identities
\begin{align}\label{binomialTh1}
\s \frac{z^n}{(q;q)_n}= \frac{1}{(z;q)_{\infty}},\, |z|<1,  \quad
\s \frac{(-1)^n q^{\binom n2}}{(q;q)_n}z^n= (z;q)_{\infty}.
\end{align}

The theory of basic hypergeometric series has been greatly developed for more than a century, and there are many effective ways to study it, such as the Wilf-Zeilberg algorithm, transformation, inversion and operator, for example, see \cite{WZ1992,Chu1993,GR,AA1993,AHD1977,A1970,ChenLiu1997,ChenLiu1998,Wang2009,C2013}. Ten years ago, Liu first introduced the $q$-partial differential equation method to study $q$-series. This novel approach has garnered attention from numerous mathematicians, refer to \cite{L2015,L2010,L2013,L2011} for further details. To this end, we initially define the $q$-partial derivative \cite{L2013}.
\begin{definition}
A $q$-partial derivative of a function of several variables is its $q$-derivative with respect to one of those variables, regarding other variables as constants.
\end{definition}
For convenience, the $q$-partial derivative of a function $f$ with respect to the variable $x$ is denoted by $\D_{q,x}\{f\}$. In \cite{L2013}, Liu  proved the following theorem.
\begin{theorem}
If $f(x,y)$ is a two-variable analytic function at $(0,0)\in \mathbb{C}^2$, then, $f$ can be expanded in terms of homogeneous Rogers-Szeg\"o polynomials (for definition see (\ref{q-Hahn})) if and only if $f$ satisfies the $q$-partial differential equation $\D_{q,x} \{f\}=\D_{q,y}\{f\}$.
\end{theorem}

We should point out that the above theorem has developed a new theory for calculating the $q$-identity and demonstrated its universality when applied to many types of $q$-orthogonal polynomials, including Rogers-Szeg\"o polynomials, Hahn polynomials, Stieltjes-Wigert polynomials and Askey-Wilson polynomials, as well as classical orthogonal polynomials such as Hermite polynomials (cf. \cite{L2017}). Later, some related works by Abdlhusein, Arjika, Cao, Jia, Li, Niu and Zhang also fall into Liu's theory. Readers interested can see \cite{NL2018,C2014,C2016,C2017,CN2016,A2020,A2016,Cao2020,Cao2021,Jia,Z2020}.

Hahn \cite{Hahn} first discovered the $q$-Laguerre polynomials, which belong to the Askey-scheme of basic hypergeometric orthogonal polynomials, according to Koekoek and Swarttouw \cite{KS1998}, they are defined by
\begin{align}\label{q-Laguerre}
\mathcal{L}^{(\alpha)}_n(x)=\frac{ (q^{\alpha+1};q)_n }{ (q;q)_n } {}_1\phi_1\left(
\begin{gathered}
q^{-n} \\ q^{ \alpha+1 }
\end{gathered}
;\,q, -q^{n+\alpha+1}x \right),  \alpha>-1.
\end{align}

Askey pointed out \cite{Askey1986} that the $q$-Laguerre polynomials converge to the Stieltjes-Wigert polynomials for $\alpha\to \infty$ thus the $q$-Laguerre polynomials are sometimes called the generalized Stieltjes-Wigert polynomials \cite{KS1998}. The explicit form of $q$-Laguerre polynomials can write as
\begin{align}\label{q-Laguerre(x)}
\mathcal{L}^{(\alpha)}_n(x)=\frac{ (q^{\alpha+1};q)_n }{ (q;q)_n }
\sum_{k=0}^n \begin{bmatrix} n \\k \end{bmatrix}
\frac{ q^{k^2+k \alpha} }{ (q^{\alpha+1};q)_k  } (-x)^k.
\end{align}

To study $q$-Laguerre polynomials from the perspective of $q$-partial differential equations following Liu's ideas, it is necessary to introduce homogeneous $q$-Laguerre polynomials
\begin{align}\label{poly1}
L_n^{(\alpha)}(x,y)=\sk \begin{bmatrix} n \\k \end{bmatrix}
\frac{  q^{ k^2+k\alpha } }{ (q^{\alpha+1};q)_k } (-x)^k y^{n-k},  \alpha>-1.
\end{align}
Obviously,
\begin{align*}
L_n^{(\alpha)}(x,y)=\frac{ (q;q)_n }{ (q^{\alpha+1};q)_n } y^n\mathcal{L}^{(\alpha)}_n (x/y), \,\,
L_n^{(\alpha)}(x,1)=\frac{ (q;q)_n }{ (q^{\alpha+1};q)_n } \mathcal{L}^{(\alpha)}_n(x), \,\,
L_n^{(\alpha)}(0,y)=y^n.
\end{align*}

This paper is organized as follows. Section \ref{Sec2} shows that if an analytic function satisfies a system of $q$-partial differential equations, if and only if it can be expanded in terms of homogeneous $q$-Laguerre polynomials (refer to Theorem \ref{thLn}). Section \ref{Sec3} is an application of Theorem \ref{thLn}, where we use the method of $q$-partial differential equations to prove the generating functions of homogeneous $q$-Laguerre polynomials with different weights. Section \ref{Sec4} presents that an analytic function can be expanded in terms of homogeneous little $q$-Jacobi polynomials (see Theorem \ref{th2}) if and only if it satisfies a system of $q$-partial differential equations. In section \ref{Sec6}, we obtain some identities by applying Theorems \ref{thLn} and \ref{th2}, which generalize famous formulas such as Ramanujan $q$-beta integrals and Andrews-Askey integrals. Section \ref{Sec5} will provide an operator representation of $q$-Laguerre polynomials, through which we can obtain bilinear generating functions for $q$-Laguerre polynomials.

\section{Homogeneous $q$-Laguerre polynomials and $q$-partial differential equation}\label{Sec2}
Firstly, Proposition \ref{proposition1} presents an important property of homogeneous $q$-Laguerre polynomials.
\begin{proposition}\label{proposition1}
For $n\in \IN \cup \{0\}$, the homogeneous $q$-Laguerre polynomials $L_n^{(\alpha)}(x,y)$ satisfy the $q$-partial differential equation
\begin{align}\label{equation-L_n}
\D_{q,x} (1- q^{\alpha} \eta_x) \left\{ L_n^{(\alpha)}(x,y) \right\}
=-q^{\alpha+1} \eta^2_x  \D_{q,y}  \left\{ L_n^{(\alpha)}(x,y) \right\},
\end{align}
namely,
\begin{align*}
\D_{q,x} \left\{ L_n^{(\alpha)}(x,y)-q^{\alpha} L_n^{(\alpha)}(qx,y) \right\}
= -q^{\alpha+1} \D_{q,y}  \left\{  L_n^{(\alpha)}(q^2x,y) \right\}.
\end{align*}
\end{proposition}

\begin{proof}
Let LHS denote the left-hand side of the equation (\ref{equation-L_n}), and by using the formula $\D_{q,x} \{x^n\}=(1-q^n)x^{n-1}$, we can obtain
\begin{align*}
{\rm LHS}&=\D_{q,x} \left\{ \sk (-1)^k
\begin{bmatrix} n \\k \end{bmatrix}
\frac{  q^{ k^2+k\alpha } }{ (q^{\alpha+1};q)_{k-1} }x^k y^{n-k} \right\}\\
&=\skk  (-1)^k \begin{bmatrix} n \\k \end{bmatrix} (1-q^k)
\frac{  q^{ k^2+k\alpha } }{ (q^{\alpha+1};q)_{k-1} }x^{k-1} y^{n-k}.
\end{align*}
Similarly, use RHS to denote the right-hand side of the equation (\ref{equation-L_n}). Through simple calculation, we obtain
\begin{align*}
{\rm RHS} &=-q^{\alpha+1} \D_{q,y} \left\{ \sk (-1)^k
\begin{bmatrix} n \\k \end{bmatrix}
\frac{  q^{ k^2+k\alpha } }{ (q^{\alpha+1};q)_k }
(q^2x)^k y^{n-k} \right\} \\
&=\sum_{k=0}^{n-1}  (-1)^{k+1}
\begin{bmatrix} n \\k \end{bmatrix} (1-q^{n-k})
\frac{  q^{ (k+1)^2+(k+1)\alpha } }{ (q^{\alpha+1};q)_k }
x^k y^{n-k-1}  \\
&=\skk (-1)^k \begin{bmatrix} n \\k-1 \end{bmatrix} (1-q^{n-k+1})
\frac{  q^{ k^2+k\alpha } }{ (q^{\alpha+1};q)_{k-1} } x^{k-1} y^{n-k}.
\end{align*}
From the definition of the $q$-binomial coefficients (\ref{binomial}), it is easy to verify that
\begin{align}\label{eq4}
\begin{bmatrix} n \\k \end{bmatrix}  (1-q^k)
=\begin{bmatrix} n \\k-1 \end{bmatrix} (1-q^{n-k+1}).
\end{align}
It follows from (\ref{eq4}) that ${\rm LHS=RHS}$.
\end{proof}
In order to prove Theorem \ref{thLn}, we need the following proposition (for example,
see \cite[p.5]{M1984}).
\begin{proposition}\label{pro}
If $f(x_1, x_2, \cdots, x_k)$ is analytic at the origin $(0, 0, \ldots, 0)\in\mathbb{C}^k$, then, $f$ can be expanded in an absolutely and uniformly convergent power series,
\begin{align*}
f(x_1, x_2, \ldots, x_k)
=\sum_{n_1, n_2, \ldots, n_k=0}^{\infty}
\la_{n_1,n_2, \ldots,n_k} x_1^{n_1}x_2^{n_2}\cdots x_k^{n_k}.
\end{align*}
\end{proposition}
The principal result of this section is the following expansion theorem.
\begin{theorem}\label{thLn}
If $f(x_1, y_1,\cdots, x_k, y_k)$ is a $2k$-variable analytic function at $(0, 0, \cdots, 0)\in \mathbb{C}^{2k}$, then, $f$ can be expanded
\begin{align*}
\sum_{n_1,\cdots,n_k=0}^{\infty} \la_{n_1,\cdots,n_k} \,
L_{n_1}^{(\alpha_1)}(x_1,y_1) \cdots L_{n_k}^{(\alpha_k)}(x_k,y_k),
\end{align*}
where $\la_{n_1,\cdots,n_k}$ are independent of $x_1,y_1,\cdots,x_k,y_k$, if and only if $f$ satisfies the $q$-partial differential equations
\begin{align}\label{Eq2}
\D_{q,x_j} (1- q^{\alpha_j} \eta_{x_j}) \left\{ f \right\}
=-q^{\alpha_j+1} \eta^2_{x_j} \D_{q,y_j}   \left\{ f \right\}
\end{align}
for $j\in \{1,2,\ldots,k\}$.
\end{theorem}

\begin{proof}
We employ mathematical induction. When $k=1$, it follows from Proposition \ref{pro} that $f$ can be expanded in an absolutely and uniformly convergent power series in a neighborhood of $(0,0)$. Therefore, there exists a sequence $\{\la_{m,n}\}$ independent of $x_1$ and $y_1$ for which
\begin{align}\label{f(x1,y1)1}
f(x_1,y_1)=\sum_{m,n=0}^{\infty} \la_{m,n} x_1^m y_1^n=\sum_{m=0}^{\infty} x_1^m  \s \la_{m,n} y_1^n .
\end{align}
Substituting the above equation into the following $q$-partial differential equation
\begin{align}\label{Eq3}
\D_{q,x_1} (1- q^{\alpha_1} \eta_{x_1}) \left\{ f(x_1,y_1) \right\}
=-q^{\alpha_1+1} \eta^2_{x_1} \D_{q,y_1}   \left\{ f(x_1,y_1) \right\},
\end{align}
we obtain
\begin{align}\label{Eq1}
\sum_{m=1}^{\infty} (1-q^{\alpha_1+m} )(1-q^m) x_1^{m-1}  \s \la_{m,n} y_1^n
=-q^{\alpha_1+1} \sum_{m=0}^{\infty} q^{2m} x_1^m
\D_{q,y_1} \left\{ \s \la_{m,n} y_1^n \right\}.
\end{align}
Equating the coefficients of $x_1^{m-1}$ in (\ref{Eq1}), we have
\begin{align*}
 \s \la_{m,n} y_1^n = \frac{ (-q^{\alpha_1+1}) q^{2(m-1)} }
 { (1-q^{\alpha_1+m} )(1-q^m) }
 \D_{q,y_1} \left\{ \s \la_{m-1,n} y_1^n \right\}.
\end{align*}
Iteration $m-1$ times yields
\begin{align*}
 \s \la_{m,n} y_1^n &= \frac{ (-q^{\alpha_1+1})^m q^{m(m-1)} }
 { (q;q)_m (q^{\alpha_1+1};q)_m }
 \D_{q,y_1}^m \left\{ \s \la_{0,n} y_1^n \right\} \\
 &=\frac{ (-1)^m  q^{m^2+m\alpha_1} } { (q;q)_m (q^{\alpha_1+1};q)_m }
  \s \la_{0,n} \frac{ (q;q)_n }{(q;q)_{n-m}} y_1^{n-m} \\
 &=\sum_{n=m}^{\infty} (-1)^m \la_{0,n} \begin{bmatrix} n \\m \end{bmatrix}
   \frac{ q^{m^2+m\alpha_1} }{ (q^{\alpha_1+1};q)_m } y_1^{n-m}.
\end{align*}
Noting that the series in (\ref{f(x1,y1)1}) is an absolutely and uniformly convergent series, substituting the above equation into (\ref{f(x1,y1)1}) and interchanging the order of the summation, we obtain
\begin{align*}
f(x_1,y_1)&=\sum_{m=0}^{\infty} x_1^m \sum_{n=m}^{\infty} (-1)^m \la_{0,n}  \begin{bmatrix} n \\m \end{bmatrix}
   \frac{ q^{m^2+m\alpha_1} }{ (q^{\alpha_1+1};q)_m } y_1^{n-m} \\
&=\s \la_{0,n} \sum_{m=0}^n  (-1)^m \begin{bmatrix} n \\m \end{bmatrix}
   \frac{ q^{m^2+m\alpha_1} }{ (q^{\alpha_1+1};q)_m }  x_1^m y_1^{n-m} \\
&=\s \la_{0,n} L_n^{(\alpha_1)}(x_1,y_1).
\end{align*}
The above calculation shows that the sufficiency of Theorem \ref{thLn} is correct. Conversely, if $f(x_1,y_1)$ can be expanded in terms of $L_n^{(\alpha_1)}(x_1,y_1)$, then using Proposition \ref{proposition2}, we find that $f(x_1,y_1)$ satisfies (\ref{Eq2}). So we can prove the case of $k=1$.

Now, we assume that Theorem \ref{thLn} is true for the case $k-1$. Since $f$ is analytic at $(0,0)$ and satisfies (\ref{Eq3}). Thus, there exists a sequence $\{c_{n_1}(x_2,y_2,\ldots,x_k,y_k)\}$ independent of $x_1$ and $y_1$ for which
\begin{align}\label{Eq4}
f(x_1,y_1,\ldots,x_k,y_k)=\sum_{n_1=0}^{\infty}
c_{n_1}(x_2,y_2,\ldots,x_k,y_k)  L_{n_1}^{(\alpha_1)}(x_1,y_1).
\end{align}
Putting $x_1=0$ in (\ref{Eq4})  and using $ L_{n_1}^{(\alpha_1)}(0,y_1)=y_1^{n_1}$, we obtain
\begin{align*}
f(0,y_1,x_2,y_2, \ldots, x_k,y_k)=\sum_{n_1=0}^{\infty}
c_{n_1}(x_2,y_2,\ldots,x_k,y_k) y_1^{n_1}.
\end{align*}
Using the Maclaurin expansion theorem, we have
\begin{align*}
c_{n_1}(x_2,y_2,\ldots,x_k,y_k)
=\frac{\partial^{n_1} f(0,y_1,x_2,y_2,\ldots,x_k,y_k)}
{ n_1! \partial {y_1}^{n_1}} \Big|_{y_1=0}.
\end{align*}
Since $f(x_1,y_1,\ldots,x_k,y_k)$ is analytic near $(x_1,y_1,\ldots,x_k,y_k)=(0,\ldots,0)\in \mathbb{C}^{2k}$, it follows from the above equation that $c_{n_1}(x_2, y_2, \ldots, x_k, y_k)$ is analytic near $(x_2,y_2,\ldots,x_k,y_k)=(0,\ldots,0)\in \mathbb{C}^{2k-2}$. Substituting (\ref{Eq4}) into (\ref{Eq2}), we find that for $j=2,\ldots,k$,
\begin{align*}
 & \sum_{n_1=0}^{\infty}  \D_{q,x_j} (1- q^{\alpha_j} \eta_{x_j}) \left\{
c_{n_1}(x_2,y_2,\ldots,x_k,y_k) \right\} L_n^{(\alpha_1)}(x_1,y_1)\\
&=\sum_{n_1=0}^{\infty}
(-q^{\alpha_j+1} \eta^2_{x_j}) \D_{q,y_j}
 \left\{c_{n_1}(x_2,y_2,\ldots,x_k,y_k) \right\}
 L_n^{(\alpha_1)}(x_1,y_1).
\end{align*}
By equating the coefficients of $L_n^{(\alpha_1)}(x_1,y_1)$ in the above equation, we obtain
\begin{align*}
\D_{q,x_j} (1- q^{\alpha_j} \eta_{x_j}) \left\{
c_{n_1}(x_2,y_2,\ldots,x_k,y_k) \right\}
=-q^{\alpha_j+1} \eta^2_{x_j} \D_{q,y_j}
 \left\{c_{n_1}(x_2,y_2,\ldots,x_k,y_k) \right\}.
\end{align*}
Therefore, there exists a sequence $\la_{n_1,n_2,\ldots,n_k}$ independent of
 $x_2, y_2, \ldots, x_k, y_k$ (of course independent of $x_1$ and $y_1$) for which
\begin{align*}
c_{n_1}(x_2,y_2,\ldots,x_k,y_k)
=\sum_{n_2,\ldots,n_k=0}^{\infty} \la_{n_1,n_2,\ldots,n_k}
L_{n_2}^{(\alpha_2)}(x_2,y_2) \cdots L_{n_k}^{(\alpha_k)}(x_k,y_k).
\end{align*}
Then substituting this equation into (\ref{Eq4}), we proved the sufficiency of Theorem \ref{thLn}. Conversely, if $f$ can be expanded in terms of $L_{n_1}^{(\alpha_1)}(x_1,y_1) \cdots L_{n_k}^{(\alpha_k)}(x_k,y_k)$, it follows from Proposition \ref{proposition1} that $f$ satisfies (\ref{Eq2}). This completes the proof.
\end{proof}

\begin{remark}
Theorem \ref{thLn} shows that an analytic function can be expanded in terms of homogeneous $q$-Laguerre polynomials, and its essence is to satisfy  $q$-partial differential equation (\ref{Eq2}). Theorem \ref{thLn} implies that all solutions to $q$-partial differential equation (\ref{Eq2}) can be represented as linear combinations of homogeneous $q$-Laguerre polynomials. This theorem is a powerful tool for proving formulas involving the homogeneous $q$-Laguerre polynomials. Its applications are discussed in Sections \ref{Sec3} and \ref{Sec6}.
\end{remark}

\section{Generating functions for homogeneous $q$-Laguerre polynomials}\label{Sec3}

Since the Stieltjes and Hamburger moment problems corresponding to the $q$-Laguerre polynomials are indeterminate there exist many different weight functions, see \cite{Anne1994,Ismail1998,C2003,Moak1981} for details. Theorem \ref{generating1} will use Theorem \ref{thLn} to prove the following generating functions of homogeneous $q$-Laguerre polynomials with different weights. We often refer to the following theorem
(see \cite[p. 28]{Taylor2002}) to determine if a given function is an analytic function in several complex variables.
\begin{theorem}\label{Hartog}
(Hartog's theorem) If a complex valued function $f(z_1,z_2,\cdots,z_n)$
is holomorphic (analytic) in each variable separately in a domain $U\in \mathbb{C}^n$, then, it is holomorphic (analytic) in $U$.
\end{theorem}

\begin{theorem}\label{generating1}
(1) We have
\begin{align}\label{Eq0}
\s \frac{(-1)^n q^{\binom n2} } {(q;q)_n} L_n^{(\alpha)}(x,y) t^n
=(ty;q)_{\infty}\, {}_0\phi_2\left(
\begin{gathered} - \\ q^{ \alpha+1 },  ty \end{gathered}
;\,q, -q^{\alpha+1} tx \right).
\end{align}
(2) For arbitrarily given $\gamma$, and for $|ty|<1$, we have
\begin{align}\label{Eq5}
\s \frac{(\gamma ;q)_n} {(q;q)_n} L_n^{(\alpha)}(x,y) t^n
=\frac{ (\gamma ty;q)_{\infty} }{ (ty;q)_{\infty} } {}_1\phi_2\left(
\begin{gathered}
\gamma \\ q^{ \alpha+1 }, \gamma ty
\end{gathered}
;\,q, -q^{\alpha+1} tx \right).
\end{align}
\end{theorem}

\begin{proof}
For part (1), denote the right-hand side of (\ref{Eq0}) by $f(x,y)$. It follows from Theorem \ref{Hartog} that $f(x,y)$ is an analytic function of $x$ and $y$. Thus $f(x,y)$ is analytic at $(0,0)\in \mathbb{C}^2$. On the one hand, we have
\begin{align*}
\D_{q,x} (1- q^{\alpha} \eta_x) \left\{ f(x,y) \right\}
=-tq^{\alpha+1}(ty;q)_{\infty}
\s \frac{ [(-1)^{n+1} q^{\binom {n+1}{2}}]^3 }
  { (q^{\alpha+1};q)_n (ty;q)_{n+1}  (q;q)_n } (-q^{\alpha+1} xt)^n.
\end{align*}
On the other hand, according to (\ref{Dfg}),
\begin{align*}
\D_{q,y}  \left\{ f(x,y) \right\}=(ty;q)_{\infty} \s \frac{ [(-1)^n q^{\binom n2}]^3 (-tq^n)}
  { (q^{\alpha+1};q)_n (ty;q)_{n+1}  (q;q)_n } (-q^{\alpha+1} xt)^n,
\end{align*}
from which we obtain
\begin{align*}
-q^{\alpha+1} \eta^2_x  \D_{q,y}   \left\{ f(x,y) \right\}
&=tq^{\alpha+1} (ty;q)_{\infty} \s \frac{ [(-1)^n q^{\binom n2}]^3 q^{3n} }
  { (q^{\alpha+1};q)_n (ty;q)_{n+1}  (q;q)_n } (-q^{\alpha+1} xt)^n \\
&=\D_{q,x}  (1- q^{\alpha} \eta_x) \left\{ f(x,y) \right\}.
\end{align*}
Then by Theorem \ref{thLn}, there exists a sequence $\{\la_n\}$ independent of $x$ and $y$ such that
\begin{align}\label{3eq2}
(ty;q)_{\infty}\, {}_0\phi_2\left( \begin{gathered} - \\ q^{ \alpha+1 },  ty \end{gathered} ;\,q, -q^{\alpha+1} xt \right) =\s \la_n L_n^{(\alpha)}(x,y) .
\end{align}
Putting $x=0$ in the above equation, using $L_n^{(\alpha)}(0,y)=y^n$
and (\ref{binomialTh1}), we find that
\begin{align*}
\s \la_n y^n=(ty;q)_{\infty}=\s \frac{(-1)^n q^{\binom n2}}{(q;q)_n}(ty)^n.
\end{align*}
Equating the coefficients of $y^n$ in the above equation, we have $\la_n=(-1)^n q^{\binom n2}/[t^n(q;q)_n]$. Then substitute it into (\ref{3eq2}) and equation (\ref{Eq0}) follows.

For part (2), denote the right-hand side of (\ref{Eq5}) by $f(x,y)$. It follows from Theorem \ref{Hartog} that $f(x,y)$ is an analytic function of $x$ and $y$ for $|ty|<1$. Thus $f(x,y)$ is analytic at $(0,0)\in \mathbb{C}^2$. On the one hand, we have
\begin{align*}
&\quad \D_{q,x}  (1- q^{\alpha} \eta_x) \left\{ f(x,y) \right\} \\
&=\D_{q,x}   \left\{ \frac{(\gamma ty;q)_{\infty} }{ (ty;q)_{\infty} }
  \s \frac{ (\gamma ;q)_n [(-1)^n q^{\binom n2}]^2 }
   { (q^{\alpha+1};q)_{n-1} (\gamma ty;q)_n (q;q)_n }
  (-q^{\alpha+1} xt)^n \right\} \\
&=\frac{-tq^{\alpha+1} (\gamma ty;q)_{\infty} }{ (ty;q)_{\infty} } \s
\frac{ (\gamma ;q)_{n+1} [(-1)^n q^{\binom n2}]^2 q^{2n} }
  { (q^{\alpha+1};q)_n (\gamma ty;q)_{n+1}  (q;q)_n } (-q^{\alpha+1} xt)^n.
\end{align*}
On the other hand, according to (\ref{Dfg}),
\begin{align*}
\D_{q,y}  \left\{ f(x,y) \right\}=\frac{ t(\gamma ty;q)_{\infty} }{ (ty;q)_{\infty} }
  \s \frac{ (\gamma ;q)_{n+1} [(-1)^n q^{\binom n2}]^2 }
  {(q^{\alpha+1};q)_n (\gamma ty;q)_{n+1} (q;q)_n } (-q^{\alpha+1} xt)^n,
\end{align*}
from which we obtain
\begin{align*}
-q^{\alpha+1} \eta^2_x  \D_{q,y}   \left\{ f(x,y) \right\}
&=\frac{-tq^{\alpha+1} (\gamma ty;q)_{\infty} }{ (ty;q)_{\infty} } \s
\frac{ (\gamma ;q)_{n+1} [(-1)^n q^{\binom n2}]^2 q^{2n} }
  { (q^{\alpha+1};q)_n (\gamma ty;q)_{n+1}  (q;q)_n } (-q^{\alpha+1} xt)^n\\
&=\D_{q,x}  (1- q^{\alpha} \eta_x) \left\{ f(x,y) \right\}.
\end{align*}
Then by Theorem \ref{thLn}, there exists a sequence $\{\la_n\}$ independent of $x$ and $y$ such that
\begin{align}\label{3eq3}
\frac{ (\gamma ty;q)_{\infty} }{ (ty;q)_{\infty} } {}_1\phi_2\left(
\begin{gathered} \gamma \\ q^{ \alpha+1 }, \gamma ty
\end{gathered} ;\,q, -q^{\alpha+1} xt \right)
=\s \la_n L_n^{(\alpha)}(x,y) .
\end{align}
Putting $x=0$ in the above equation, using $L_n^{(\alpha)}(0,y)=y^n$
and (\ref{binomialTh}), we find that
\begin{align*}
\s \la_n y^n=\frac{ (\gamma ty;q)_{\infty} }{ (ty;q)_{\infty} }
=\s \frac{(\gamma;q)_n}{(q;q)_n}(ty)^n.
\end{align*}
Equating the coefficients of $y^n$ in the above equation, we obtain $\la_n=t^n (\gamma;q)_n/(q;q)_n$. Then substitute it into (\ref{3eq3}), which completes the proof of (\ref{Eq5}).
\end{proof}

\begin{remark}
Taking $\gamma=0$ in (\ref{Eq5}), we can obtain a simpler generating function for $L_n^{(\alpha)}(x,y)$:
\begin{align}\label{3Eq1}
\s \frac{L_n^{(\alpha)}(x,y)} {(q;q)_n} t^n
=\frac{1}{ (ty;q)_{\infty} } {}_0\phi_1\left(
\begin{gathered}
- \\ q^{ \alpha+1 }
\end{gathered}
;\,q, -q^{\alpha+1} xt \right).
\end{align}
\end{remark}

\section{Homogeneous little $q$-Jacobi polynomials and $q$-partial differential equation}\label{Sec4}

A $q$-analogue of Jacobi polynomials was introduced by Hahn \cite{Hahn} and
later studied by Andrews and Askey \cite{AA1977,AA1985}, and named by them as little $q$-Jacobi polynomials:
\begin{align}\label{q-Jacobi}
\mathcal{P}_n^{(\alpha,\beta)}(x)={}_2\phi_1\left( \begin{gathered}
q^{-n},\alpha\beta q^{n+1} \\ \alpha q \end{gathered} ;\,q,qx \right).
\end{align}
As $q\to1$, the little $q$-Jacobi polynomials tend to a multiple of Jacobi polynomials. Little $q$-Jacobi polynomials with $\beta=0$ are $q$-analogs of Laguerre polynomials and are orthogonal with respect to a discrete measure on a countable set, called little $q$-Laguerre (or Wall) polynomials. Moreover, the little $q$-Legendre polynomials are little $q$-Jacobi polynomials with $\alpha=\beta=1$. If we set $\beta\rightarrow -\alpha^{-1}q^{-1}\beta$, in the little $q$-Jacobi polynomials and then take the limit $\alpha \rightarrow 0$ we obtain the alternative $q$-Charlier polynomials. For more details about $q$-Jacobi polynomials, see \cite{KS1998}.

To establish the relationship between little $q$-Jacobi polynomials and $q$-partial differential equations, similar to Section \ref{Sec2}, we naturally introduce homogeneous little $q$-Jacobi polynomials
\begin{align}\label{poly2}
p_n^{(\alpha,\beta)}(x,y)=\sk  q^{ k(k+1-2n)/2 }\begin{bmatrix} n \\k \end{bmatrix}
\frac{ (\alpha\beta q^{n+1};q)_k }{ (\alpha q;q)_k }(-x)^k y^{n-k}.
\end{align}
Evidently,
\begin{align}\label{poly3}
p_n^{(\alpha,\beta)}(x,y)=y^n \mathcal{P}_n^{(\alpha,\beta)}(x/y),\,\,
p_n^{(\alpha,\beta)}(x,1)=\mathcal{P}_n^{(\alpha,\beta)}(x), \,\,
p_n^{(\alpha,\beta)}(0,y)=y^n.
\end{align}
Firstly, Proposition \ref{proposition2} shows an important property of homogeneous little $q$-Jacobi polynomials.
\begin{proposition}\label{proposition2}
The homogeneous little $q$-Jacobi polynomials $p_n^{(\alpha,\beta)}(x,y)$ satisfy the $q$-partial differential equation
\begin{align}\label{equation1}
\D_{q,x}  (1-\alpha \eta_x) \left\{ p_n^{(\alpha,\beta)}(x,y) \right\}
=-q \D_{q,y}   (\eta^{-1}_y-q \alpha \beta \eta^2_x )
\left\{ p_n^{(\alpha,\beta)}(x,y) \right\},
\end{align}
namely,
\begin{align*}
 \D_{q,x}  \left\{ p_n^{(\alpha,\beta)}(x,y)
-\alpha p_n^{(\alpha,\beta)}(qx,y) \right\}
= -q \D_{q,y}   \left\{  p_n^{(\alpha,\beta)}(x,y/q)
-q \alpha \beta p_n^{(\alpha,\beta)}(q^2x,y)  \right\}.
\end{align*}

\end{proposition}

\begin{proof}
If we use LHS to denote the left-hand side of the equation (\ref{equation1}), we have
\begin{align*}
{\rm LHS} &=\D_{q,x}  \left\{ \sk (-1)^k  q^{ k(k+1-2n)/2 }
\begin{bmatrix} n \\k \end{bmatrix}
\frac{ (\alpha\beta q^{n+1};q)_k }{ (\alpha q;q)_{k-1} }x^k y^{n-k}\right\}\\
&=\skk (-1)^k  q^{ k(k+1-2n)/2 }
\begin{bmatrix} n \\k \end{bmatrix} (1-q^k)
\frac{ (\alpha\beta q^{n+1};q)_k }{ (\alpha q;q)_{k-1} }x^{k-1} y^{n-k}.
\end{align*}
Similarly, use RHS to denote the right-hand side of the equation (\ref{equation1}). By simple calculation, we obatin
\begin{align*}
{\rm RHS} &=\D_{q,y}  \left\{ \sk (-1)^{k+1}  q^{ (k+1)(k+2-2n)/2 }
\begin{bmatrix} n \\k \end{bmatrix}
\frac{ (\alpha\beta q^{n+1};q)_{k+1} }
{ (\alpha q;q)_k }x^k y^{n-k} \right\} \\
&=\sum_{k=0}^{n-1} (-1)^{k+1}  q^{ (k+1)(k+2-2n)/2 }
\begin{bmatrix} n \\k \end{bmatrix} (1-q^{n-k})
\frac{ (\alpha\beta q^{n+1};q)_{k+1} }
{ (\alpha q;q)_k }x^k y^{n-k-1} \\
&=\skk (-1)^k  q^{ k(k+1-2n)/2 }
\begin{bmatrix} n \\k-1 \end{bmatrix} (1-q^{n-k+1})
\frac{ (\alpha\beta q^{n+1};q)_k }
{ (\alpha q;q)_{k-1} }x^{k-1} y^{n-k}.
\end{align*}
It follows from (\ref{eq4}) that ${\rm LHS=RHS}$.
\end{proof}
The main result of this section is the expansion theorem presented below.
\begin{theorem}\label{th2}
If $f(x_1, y_1,\cdots, x_k, y_k)$ is a $2k$-variable analytic function at $(0, 0, \cdots, 0)\in \mathbb{C}^{2k}$, then, $f$ can be expanded
\begin{align*}
\sum_{n_1,\cdots,n_k=0}^{\infty} \la_{n_1,\cdots,n_k} \,
p_{n_1}^{(\alpha_1,\beta_1)}(x_1,y_1) \cdots p_{n_k}^{(\alpha_k,\beta_k)}(x_k,y_k),
\end{align*}
where $\la_{n_1,\cdots,n_k}$ are independent of $x_1,y_1,\cdots,x_k,y_k$,
if and only if $f$ satisfies the $q$-partial differential equations
\begin{align}\label{eq13}
\D_{q,x_j} (1-\alpha_j \eta_{x_j}) \left\{ f \right\}
=-q \D_{q,y_j} (\eta^{-1}_{y_j}-q \alpha_j \beta_j \eta^2_{x_j})
\left\{ f \right\}
\end{align}
for $j\in \{1,2,\ldots,k\}$.
\end{theorem}

\begin{proof}
We use mathematical induction. When $k=1$, it follows from Proposition \ref{pro} that $f$ can be expanded in an absolutely and uniformly convergent power series in a neighborhood of $(0,0)$. Therefore, there exists a sequence $\{\la_{m,n}\}$ independent of $x_1$ and $y_1$ for which
\begin{align}\label{f(x1,y1)}
f(x_1,y_1)=\sum_{m,n=0}^{\infty} \la_{m,n} x_1^m y_1^n=\sum_{m=0}^{\infty} x_1^m  \s \la_{m,n} y_1^n.
\end{align}
Substituting the above equation into following $q$-partial differential equation
\begin{align}\label{eq5}
\D_{q,x_1} (1-\alpha_1 \eta_{x_1} ) \left\{ f(x_1,y_1) \right\}
=-q \D_{q,y_1} (\eta^{-1}_{y_1}-q \alpha_1 \beta_1 \eta^2_{x_1})
\left\{ f(x_1,y_1) \right\}.
\end{align}
The left-hand side of (\ref{eq5}) can be written as
\begin{align*}
 \sum_{m=1}^{\infty} (1-\alpha_1 q^m)(1-q^m) x_1^{m-1}  \s \la_{m,n} y_1^n
=\sum_{m=0}^{\infty} (1-\alpha_1 q^{m+1})(1-q^{m+1})
x_1^m  \s \la_{m+1,n} y_1^n,
\end{align*}
and right-hand side of (\ref{eq5}) can be expressed as
\begin{align*}
&\quad \D_{q,y_1} \left\{ \sum_{m=0}^{\infty} \s
(-q)(q^{-n}-\alpha_1 \beta_1 q^{2m+1}) \la_{m,n} x_1^m  y_1^n \right\} \\
&= \sum_{m=0}^{\infty} \s (-q) (1-q^n) (q^{-n}-\alpha_1 \beta_1 q^{2m+1})
 \la_{m,n} x_1^m  y_1^{n-1}.
\end{align*}
Therefore, we obtain
\begin{align}\label{eq7}
&\quad \sum_{m=0}^{\infty} (1-\alpha_1 q^{m+1})(1-q^{m+1})
x_1^m  \s \la_{m+1,n} y_1^n \nonumber \\
& = \sum_{m=0}^{\infty} \s (-q) (1-q^n) (q^{-n}-\alpha_1 \beta_1 q^{2m+1})
 \la_{m,n} x_1^m  y_1^{n-1}.
\end{align}
Equating the coefficients of $x_1^m$ in (\ref{eq7}), we easily see that
\begin{align*}
&(1-q^m)(1-\alpha_1 q^m) \s \la_{m,n} y_1^n \\
&=-q \s (1-q^{n+1}) (q^{-(n+1)}-\alpha_1 \beta_1 q^{2(m-1)+1})
\la_{m-1,n+1}  y_1^n.
\end{align*}
It follows from recurrence relation of the above equation, we obtain
\begin{align}
& ( 1-q^{m-1} )(1-\alpha_1 q^{m-1}) \s \la_{m-1,n} y_1^n \nonumber  \\
& =-q \s (1-q^{n+1}) (q^{-(n+1)}-\alpha_1 \beta_1 q^{2(m-2)+1})
  \la_{m-2,n+1}  y_1^n, \label{eq12} \\
&( 1-q^{m-2} )(1-\alpha_1 q^{m-2}) \s \la_{m-2,n} y_1^n  \nonumber  \\
&=-q \s (1-q^{n+1}) (q^{-(n+1)}-\alpha_1 \beta_1 q^{2(m-3)+1}) \la_{m-3,n+1}  y_1^n, \\
& \qquad \qquad \qquad \qquad \vdots  \nonumber \\
& ( 1-q^2 )(1-\alpha_1 q^2) \s \la_{2,n} y_1^n  \nonumber  \\
& =-q \s (1-q^{n+1}) (q^{-(n+1)}-\alpha_1 \beta_1 q^{2\cdot 1+1})
\la_{1,n+1}  y_1^n,  \label{eq11} \\
& ( 1-q )(1-\alpha_1 q) \s \la_{1,n} y_1^n \nonumber  \\
&=-q \s (1-q^{n+1}) (q^{-(n+1)}-\alpha_1 \beta_1 q^{2\cdot 0+1}) \la_{0,n+1}  y_1^n. \label{eq10}
\end{align}
Now, equating the coefficients of $y_1^n$ on both sides of (\ref{eq12})-(\ref{eq10}), we easily deduce that
\begin{align*}
\la_{m,n}&=\frac{-q(1-q^{n+1})(q^{-(n+1)}-\alpha_1 \beta_1 q^{2(m-1)+1})}
{(1-q^m)(1-\alpha_1 q^m)}\la_{m-1,n+1},  \\
\la_{m-1,n}&=\frac{-q(1-q^{n+1})(q^{-(n+1)}-\alpha_1 \beta_1 q^{2(m-2)+1})}
{(1-q^{m-1}) (1-\alpha_1 q^{m-1})}\la_{m-2,n+1}, \\
&\qquad \qquad \qquad \qquad  \vdots \\
\la_{2,n}&=\frac{-q(1-q^{n+1})(q^{-(n+1)}-\alpha_1 \beta_1 q^{2\cdot1+1})}
{(1-q^2) (1-\alpha_1 q^2)}\la_{1,n+1}, \\
\la_{1,n}&=\frac{-q(1-q^{n+1})(q^{-(n+1)}-\alpha_1 \beta_1 q^{2\cdot0+1})}
{(1-q) (1-\alpha_1 q)}\la_{0,n+1}.
\end{align*}
Next, iterating the above equations $m-1$ times, we conclude that
\begin{align*}
\la_{m,n}&=\frac{-q(1-q^{n+1})(q^{-(n+1)}-\alpha_1 \beta_1 q^{2(m-1)+1})}
{(1-q^m)(1-\alpha_1 q^m)} \\
& \quad \times \frac{-q(1-q^{n+2})(q^{-(n+2)}-\alpha_1 \beta_1 q^{2(m-2)+1})}
{(1-q^{m-1}) (1-\alpha_1 q^{m-1})} \\
&\qquad\qquad\qquad\qquad\qquad  \vdots  \\
& \quad \times \frac{-q(1-q^{n+m-1})(q^{-(n+m-1)}
-\alpha_1 \beta_1 q^{2\cdot1+1})}{(1-q^2) (1-\alpha_1 q^2)} \\
& \quad \times \frac{-q(1-q^{n+m})(q^{-(n+m)}
-\alpha_1 \beta_1 q^{2\cdot0+1})} {(1-q) (1-\alpha_1 q)}\la_{0,n+m} \\
&=\frac{(-q)^m (q^{n+1};q)_m}{(q;q)_m (\alpha_1 q;q)_m}
(q^{-(n+1)}-\alpha_1 \beta_1 q^{2(m-1)+1})\cdots
(q^{-(n+m)}-\alpha_1 \beta_1 q^{2\cdot0+1})\la_{0,n+m} \\
&=q^{m(1-2n-m)/2} \frac{\la_{0,n+m} (-1)^m (q;q)_{m+n}}
{(q;q)_m (q;q)_n (\alpha_1 q;q)_m}
(1-\alpha_1 \beta_1 q^{m+n+1})\cdots(1-\alpha_1 \beta_1 q^{2m+n}) \\
&=\la_{0,n+m} (-1)^m q^{m(1-2n-m)/2} \begin{bmatrix} n+m \\m \end{bmatrix}
\frac{(\alpha_1 \beta_1 q^{m+n+1};q)_m}{(\alpha_1 q;q)_m}.
\end{align*}
It follows from the above equation that
\begin{align*}
 \s \la_{m,n} y_1^n &= \s (-1)^m \la_{0,n+m} q^{m(1-2n-m)/2}
 \begin{bmatrix} n+m \\m \end{bmatrix} \frac{(\alpha_1 \beta_1 q^{m+n+1};q)_m} {(\alpha_1 q;q)_m}  y_1^n \\
&=\sum_{n=m}^{\infty} (-1)^m \la_{0,n}  q^{m(1-2n+m)/2}
 \begin{bmatrix} n \\m \end{bmatrix} \frac{(\alpha_1 \beta_1 q^{n+1};q)_m} {(\alpha_1 q;q)_m} y_1^{n-m}.
\end{align*}
Noting that the series in (\ref{f(x1,y1)}) is an absolutely and uniformly convergent series, substituting the above equation into (\ref{f(x1,y1)}) and interchanging the order of the summation, we obtain
\begin{align*}
f(x_1,y_1)&=\sum_{m=0}^{\infty}  x_1^m
 \sum_{n=m}^{\infty} \la_{0,n}  (-1)^m q^{m(1-2n+m)/2}
 \begin{bmatrix} n \\m \end{bmatrix} \frac{(\alpha_1 \beta_1 q^{n+1};q)_m} {(\alpha_1 q;q)_m}  y_1^{n-m} \\
&=\s \la_{0,n} \sum_{m=0}^n  (-1)^m q^{m(1-2n+m)/2}
 \begin{bmatrix} n \\m \end{bmatrix} \frac{(\alpha_1 \beta_1 q^{n+1};q)_m} {(\alpha_1 q;q)_m}   x_1^m y_1^{n-m} \\
&=\s \la_{0,n}\, p_n^{(\alpha_1,\beta_1)}(x_1,y_1).
\end{align*}
The above calculation shows that the sufficiency of Theorem \ref{th2} is correct. Conversely, if $f(x_1,y_1)$ can be expanded in terms of $p_n^{(\alpha_1,\beta_1)}(x_1,y_1)$, then using Proposition \ref{proposition2}, we find that $f(x_1,y_1)$ satisfies (\ref{eq5}). So we can prove the case of $k=1$.

Next, we assume that Theorem \ref{th2} is true for the case $k-1$. Since $f$ is analytic at $(0,0)$. Thus, there exists a sequence $\{c_{n_1}(x_2,y_2,\ldots,x_k,y_k)\}$ independent of $x_1$ and $y_1$ such that
\begin{align}\label{f(x_1,x_k)}
f(x_1,y_1,\ldots,x_k,y_k)=\sum_{n_1=0}^{\infty}
c_{n_1}(x_2,y_2,\ldots,x_k,y_k) p_{n_1}^{(\alpha_1,\beta_1)}(x_1,y_1).
\end{align}
Putting $x_1=0$ in (\ref{f(x_1,x_k)})  and using $p_{n_1}^{(\alpha_1,\beta_1)}(0,y_1)=y_1^{n_1}$, we obtain
\begin{align*}
f(0,y_1,x_2,y_2, \ldots, x_k,y_k)=\sum_{n_1=0}^{\infty}
c_{n_1}(x_2,y_2,\ldots,x_k,y_k) y_1^{n_1}.
\end{align*}
Using the Maclaurin expansion theorem, we have
\begin{align*}
c_{n_1}(x_2,y_2,\ldots,x_k,y_k)
=\frac{\partial^{n_1} f(0,y_1,x_2,y_2,\ldots,x_k,y_k)}
{ n_1! \partial {y_1}^{n_1}} \Big|_{y_1=0}.
\end{align*}
Since $f(x_1,y_1,\ldots,x_k,y_k)$ is analytic near $(x_1,y_1,\ldots,x_k,y_k)=(0,\ldots,0)\in \mathbb{C}^{2k}$, it follows from the above equation that $c_{n_1}(x_2, y_2, \ldots, x_k, y_k)$ is analytic near $(x_2,y_2,\ldots,x_k,y_k)=(0,\ldots,0)\in \mathbb{C}^{2k-2}$. Substituting (\ref{f(x_1,x_k)}) into (\ref{eq13}), we find that for $j=2,\ldots,k$,
\begin{align*}
 & \sum_{n_1=0}^{\infty} \D_{q,x_j} (1-\alpha \eta_{x_j}) \left\{
c_{n_1}(x_2,y_2,\ldots,x_k,y_k) \right\} p_{n_1}^{(\alpha_1,\beta_1)}(x_1,y_1) \\
&=\sum_{n_1=0}^{\infty} \left[-q\D_{q,y_j} (\eta^{-1}_{y_j}-q \alpha \beta \eta^2_{x_j})\right] \left\{ c_{n_1}(x_2,y_2,\ldots,x_k,y_k) \right\}
p_{n_1}^{(\alpha_1,\beta_1)}(x_1,y_1).
\end{align*}
By equating the coefficients of $p_{n_1}^{(\alpha_1,\beta_1)}(x_1,y_1)$ in the above equation, we obtain
\begin{align*}
 & \D_{q,x_j} (1-\alpha \eta_{x_j})
 \left\{ c_{n_1}(x_2,y_2,\ldots,x_k,y_k) \right\} \\
&=-q\D_{q,y_j} (\eta^{-1}_{y_j}-q \alpha \beta \eta^2_{x_j})
\left\{ c_{n_1}(x_2,y_2,\ldots,x_k,y_k) \right\}.
\end{align*}
Therefore, by the inductive hypothesis, there exists a sequence $\la_{n_1,n_2,\ldots,n_k}$ independent of $x_2$, $y_2$, $\ldots$, $x_k$, $y_k$ (of course independent of $x_1$ and $y_1$) for which
\begin{align*}
c_{n_1}(x_2,y_2,\ldots,x_k,y_k)
=\sum_{n_2,\ldots,n_k=0}^{\infty} \la_{n_1,n_2,\ldots,n_k}
p_{n_2}^{(\alpha_2,\beta_2)}(x_2,y_2) \cdots p_{n_k}^{(\alpha_k,\beta_k)}(x_k,y_k).
\end{align*}
Substituting this equation into (\ref{f(x_1,x_k)}), we proved the sufficiency of the theorem. Conversely, if $f$ can be expanded in terms of $p_{n_1}^{(\alpha_1,\beta_1)}(x_1,y_1) \cdots p_{n_k}^{(\alpha_k,\beta_k)}(x_k,y_k)$, it follows from (\ref{equation1}) that $f$ satisfies (\ref{eq13}). This completes the proof of Theorem \ref{th2}.
\end{proof}

\begin{remark}
Theorem \ref{th2} shows that an analytic function can be expanded in terms of homogeneous little $q$-Jacobi polynomials, and its essence is to satisfy  $q$-partial differential equation (\ref{eq13}). Theorem \ref{th2} implies that all solutions to $q$-partial differential equation (\ref{eq13}) can be represented as linear combinations of homogeneous little $q$-Jacobi polynomials. See Section \ref{Sec6} for the application of this theorem.
\end{remark}
At the end of this section, we will present the generating function of homogeneous little $q$-Jacobi polynomials.
\begin{theorem}\label{5Th1}
Generating function for homogeneous little $q$-Jacobi polynomials:
\begin{align*}
\s \frac{q^{n(n-1)/2} t^n}{(q;q)_n (\beta q;q)_n}
p_n^{(\alpha,\beta)}(x,y)
={}_0\phi_1\left(
\begin{gathered} - \\ \alpha q \end{gathered} ;\,q, -\alpha qxt \right)
{}_2\phi_1\left(
\begin{gathered} y/x, - \\ \beta q \end{gathered} ;\,q, -xt \right).
\end{align*}
\end{theorem}

\begin{proof}
It follows from \cite{KS1998} that
\begin{align*}
\s \frac{(-1)^n q^{n(n-1)/2} t^n}{(q;q)_n (\beta q;q)_n}
\mathcal{P}_n^{(\alpha,\beta)}(x)
={}_0\phi_1\left( \begin{gathered} - \\ \alpha q \end{gathered} ;\,q, \alpha qxt \right) {}_2\phi_1\left(
\begin{gathered} 1/x, - \\ \beta q \end{gathered} ;\,q, xt \right).
\end{align*}
If $x$ is replaced by $x/y$ in the above equation, we have
\begin{align*}
\s \frac{(-1)^n q^{n(n-1)/2} t^n}{(q;q)_n (\beta q;q)_n}
\mathcal{P}_n^{(\alpha,\beta)}(x/y)
={}_0\phi_1\left( \begin{gathered} - \\ \alpha q \end{gathered} ;\,q, \alpha qxt/y \right) {}_2\phi_1\left(
\begin{gathered} y/x, - \\ \beta q \end{gathered} ;\,q, xt/y \right).
\end{align*}
Letting further $t\to -ty$ in the above equation gives
\begin{align*}
\s \frac{ q^{n(n-1)/2} t^n}{(q;q)_n (\beta q;q)_n}
y^n \mathcal{P}_n^{(\alpha,\beta)}(x/y)
={}_0\phi_1\left( \begin{gathered} - \\ \alpha q \end{gathered} ;
\,q, -\alpha qxt \right) {}_2\phi_1\left(
\begin{gathered} y/x, - \\ \beta q \end{gathered} ;\,q, -xt \right).
\end{align*}
Finally, we can deduce the conclusion by combining the above equation with (\ref{poly3}).
\end{proof}
By using Proposition \ref{proposition2}, we can determine that the right-hand side of the equation in Theorem \ref{5Th1} satisfies the $q$-partial differential equation (\ref{equation1}). Therefore, we have the following Corollary \ref{5corollary}, which will be applied in Section \ref{Sec6}.
\begin{corollary}\label{5corollary}
We have
\begin{align*}
& \D_{q,x}  (1-\alpha \eta_x) \left\{ {}_0\phi_1\left(
\begin{gathered} - \\ \alpha q \end{gathered} ;\,q, -\alpha qxt \right)
{}_2\phi_1\left(
\begin{gathered} y/x, - \\ \beta q \end{gathered} ;\,q, -xt \right) \right\}
\nonumber \\
&=-q \D_{q,y}   (\eta^{-1}_y-q \alpha \beta \eta^2_x )
\left\{ {}_0\phi_1\left(
\begin{gathered} - \\ \alpha q \end{gathered} ;\,q, -\alpha qxt \right)
{}_2\phi_1\left(
\begin{gathered} y/x, - \\ \beta q \end{gathered} ;\,q, -xt \right) \right\}.
\end{align*}
\end{corollary}

\section{Applications of Theorems \ref{thLn} and \ref{th2}}\label{Sec6}

The Rogers-Szeg\"o polynomials are famous $q$-polynomials which play an essential role in the theory of orthogonal polynomials. Liu \cite{L2013} studied the homogeneous Rogers-Szeg\"o polynomials from the perspective of $q$-partial differential equations, which are defined as
\begin{align}\label{q-Hahn}
h_n(x,y)=\sum_{k=0}^n \begin{bmatrix} n \\k \end{bmatrix} x^k y^{n-k}.
\end{align}
Further, the homogeneous Hahn polynomials
\begin{align}\label{Psi}
\Psi_n^{(a)}(x,y)=\sum_{k=0}^n \begin{bmatrix} n \\k \end{bmatrix} (a;q)_k x^k y^{n-k}
\end{align}
are a generalization of homogeneous Rogers-Szeg\"o polynomials. They were first studied by Hahn \cite{Hahn}, and then by Al-Salam and Carlitz \cite{Al1965}. So they are also called Al-Salam-Carlitz polynomials. The following generating functions will be frequently used
(cf. \cite{Al1965,L2015})
\begin{align}\label{6gen1}
\s \frac{\Psi_n^{(a)}(x,y) }{(q;q)_n}t^n=\frac{(axt;q)_{\infty}}{(xt,yt;q)_{\infty}}, \,\, \max \{|xt|,|yt|\}<1.
\end{align}
We present one of the two famous Ramanujan $q$-beta integrals \cite{A1982}.
\begin{proposition}\label{6pro3}
For $m\in \mathbb{R}$, $0<q=e^{-2k^2}<1$ and $|yzq|<1$, we have
\begin{align}\label{6eq15}
&\int_{-\infty}^{+\infty} \frac{e^{-\theta^2+2m\theta}}{(yq^{1/2}e^{2ki \theta};q)_{\infty}
 (zq^{1/2}e^{-2ik\theta};q)_{\infty}} d \theta \nonumber \\
& =\sqrt{\pi} e^{m^2} \frac{(-yqe^{2mki};q)_{\infty}  (-zqe^{-2mki};q)_{\infty}}{(yzq;q)_{\infty}}.
\end{align}
\end{proposition}

The following Theorem \ref{RamanujanBeta} is a generalization of the Proposition \ref{6pro3}.
\begin{theorem}\label{RamanujanBeta}
For $m\in \mathbb{R}$ and $\alpha>-1$, $0<q=e^{-2k^2}<1$ and $|yzq|<1$, we have
\begin{align*}
 & \int_{-\infty}^{+\infty} \frac{e^{-\theta^2+2m\theta}}{(zq^{1/2}e^{-2ik\theta};q)_{\infty}}
\s \frac{L_n^{(\alpha)}(x,y) h_n(-qe^{2mki},q^{1/2}e^{2ki \theta})}{(q;q)_n} d \theta \\
&=\sqrt{\pi} e^{m^2} \frac{(-zqe^{-2mki};q)_{\infty}}{(yzq;q)_{\infty}}
 {}_0\phi_1\left( \begin{gathered} - \\ q^{ \alpha+1 } \end{gathered} ;\,q, -q^{\alpha+2} xz \right).
\end{align*}
\end{theorem}

\begin{proof}
We use $f(x,y)$ to represent the left-hand side of the equation in Theorem \ref{RamanujanBeta}. Obviously, $f(x,y)$ is analytic near $(0,0)\in \mathbb{C}^2$. It is evident from (\ref{3Eq1}) that $f(x,y)$ satisfies
\begin{align*}
 \D_{q,x} (1- q^{\alpha} \eta_x) \left\{ f(x,y) \right\}=-q^{\alpha+1} \eta^2_x  \D_{q,y}  \left\{ f(x,y) \right\}.
\end{align*}
According to Theorem \ref{thLn}, there exists a sequence $\{\la_n\}$ independent of $x$ and $y$ such that
\begin{align}\label{6eq12}
\sqrt{\pi} e^{m^2} \frac{(-zqe^{-2mki};q)_{\infty}}{(yzq;q)_{\infty}}
 {}_0\phi_1\left( \begin{gathered} - \\ q^{ \alpha+1 } \end{gathered} ;\,q, -q^{\alpha+2} xz \right)=\s \lambda_n L_n^{(\alpha)}(x,y).
\end{align}
By letting $x=0$ in the above equation and using $L_n^{(\alpha)}(0,y)=y^n$, we can derive that
\begin{align}\label{6eq10}
\sqrt{\pi} e^{m^2} \frac{(-zqe^{-2mki};q)_{\infty}}{(yzq;q)_{\infty}}=\s \lambda_n y^n.
\end{align}
Next, by using equations (\ref{6gen1}) and (\ref{6eq15}),
\begin{align}
\sqrt{\pi} e^{m^2} \frac{(-zqe^{-2mki};q)_{\infty}}{(yzq;q)_{\infty}}&=\frac{1}{ (-yqe^{2mki};q)_{\infty} }
\int_{-\infty}^{+\infty} \frac{e^{-\theta^2+2m\theta} }{ (yq^{1/2}e^{2ik\theta}, zq^{1/2}e^{-2ik\theta} ;q)_{\infty} } d \theta \nonumber \\
&=\s \int_{-\infty}^{+\infty} \frac{e^{-\theta^2+2m\theta}}{(zq^{1/2}e^{-2ik\theta} ;q)_{\infty}}
\frac{ h_n(-qe^{2mki},q^{1/2}e^{2ki \theta}) }{(q;q)_n} d\theta  y^n.  \label{6eq11}
\end{align}
Then comparing the $y^n$ coefficients of (\ref{6eq10}) and (\ref{6eq11}), we can obtain
\begin{align*}
\lambda_n=\int_{-\infty}^{+\infty} \frac{e^{-\theta^2+2m\theta}}{(zq^{1/2}e^{-2ik\theta} ;q)_{\infty}}
\frac{ h_n(-qe^{2mki},q^{1/2}e^{2ki \theta}) }{(q;q)_n} d\theta.
\end{align*}
Finally, substitute the above equation into (\ref{6eq12}) to complete the proof.
\end{proof}

\begin{remark}
When $x=0$, Theorem \ref{RamanujanBeta} degenerates into Proposition \ref{6pro3}. In the later Theorem \ref{7th3}, we will give Theorem \ref{RamanujanBeta} an equivalent form.
\end{remark}

Now, we shall present some applications of Theorems \ref{thLn} and \ref{th2} in $q$-integral. The Jackson $q$-integral of the function $f(x)$ from $a$ to $b$ is defined as
\begin{align}\label{q-integral}
\int_a^b f(x)d_qx=(1-q)\s [bf(bq^n)-af(aq^n)]q^n.
\end{align}
If $f$ is continuous on $(a,b)$, then it is easily seen that
\begin{align*}
\lim_{q\to1} \int_a^b f(x)d_qx=\int_a^b f(x)dx.
\end{align*}
The famous Andrews-Askey integral formula \cite[Theorem 1]{AA1981} can be stated in the following proposition.
\begin{proposition}\label{6pro1}
For $\max \{ |bu|,|bv|,|cu|,|cv| \}<1$, we have
\begin{align*}
\int_u^v  \frac{ (qx/u,qx/v;q)_{\infty} }{ (bx,cx;q)_{\infty} }  d_qx
=\frac{(1-q)v(q,u/v,qv/u,bcuv;q)_{\infty}}{(bu,bv,cu,cv;q)_{\infty}}.
\end{align*}
\end{proposition}
In \cite[Theorem 4.4]{L2015}, Liu extended Proposition \ref{6pro1} and proved the following $q$-integral formula.
\begin{proposition}\label{6pro2}
If there are no zero factors in the denominator of the integral, we have
\begin{align*}
& \int_u^v \frac{(qx/u,qx/v,\beta ax;q)_{\infty} }{ (ax,bx,cx,dx;q)_{\infty} } d_qx \\
& =\frac{(1-q)v (q,u/v,qv/u,cduv;q)_{\infty} }{(cu,cv,du,dv;q)_{\infty} }
\s \frac{W_n(c,d,u,v) \Psi_n^{(\beta)}(a,b)}{ (q;q)_n }
\end{align*}
with
\begin{align}\label{Wn}
W_n(a,b,u,v)=\sum_{k=0}^n \begin{bmatrix} n \\k \end{bmatrix}
\frac{ (av,bv;q)_k }{ (abuv;q)_k } u^k v^{n-k}.
\end{align}
\end{proposition}
The main results of this section is the following Theorems \ref{6th1} and \ref{6th2}.

\begin{theorem}\label{6th1}
 For $\max \{ |cu|,|cv|,|du|,|dv|,|bzu|,|bzv| \}<1$, we have
\begin{align*}
& \int_u^v  \frac{ (qx/u,qx/v;q)_{\infty} }{ (cx,dx;q)_{\infty} }
\T(\beta,y; \D_{q,z}) \left\{ \frac{1}{(bzx;q)_{\infty}} {}_0\phi_1\left( \begin{gathered} - \\ q^{ \alpha+1 } \end{gathered} ;\,q, -q^{\alpha+1} azx \right) \right\} d_qx \\
&=\frac{(1-q)v(q,u/v,qv/u,cduv;q)_{\infty}}{(cu,cv,du,dv;q)_{\infty}}
\s \frac{ W_n(c,d,u,v) \Psi_n^{(\beta)}(y,z)}
{ (q;q)_n } L_n^{(\alpha)}(a,b)
\end{align*}
with
\begin{align*}
\T(\beta,y;\D_{q,z})= \sum_{k=0}^{\infty} \frac{(\beta;q)_k}{(q;q)_k} (y\D_{q,z})^k,
\end{align*}
it is called the Cauchy augmentation operator \cite[(1.2)]{CG2008}.
\end{theorem}

\begin{proof}
We use $I(a,b)$ to represent the left-hand side of the equation in Theorem \ref{6gen1}. Therefore,
\begin{align}\label{6eq8}
I(a,b)=\T(\beta,y; \D_{q,z}) \int_u^v  \frac{ (qx/u,qx/v;q)_{\infty} }{ (cx,dx,bzx;q)_{\infty} }
{}_0\phi_1\left( \begin{gathered}- \\ q^{ \alpha+1 } \end{gathered} ;\,q, -q^{\alpha+1} azx \right) d_qx.
\end{align}
It is obvious that (\ref{6eq8}) is analytic near $(0,0)\in \mathbb{C}^2$ for $\max \{ |cu|,|cv|,|du|,|dv|,|bzu|,|bzv| \}<1$. By (\ref{3Eq1}) and (\ref{Psi}), and (\ref{6eq8}) can be rewritten as
\begin{align}\label{5eq2}
I(a,b)&= \int_u^v  \frac{ (qx/u,qx/v;q)_{\infty} }{ (cx,dx;q)_{\infty} }
\T(\beta,y; \D_{q,z}) \left\{  \s \frac{ L_n^{(\alpha)}(a,b)}{(q;q)_n} (xz)^n \right\} d_qx \nonumber \\
&=\int_u^v  \frac{ (qx/u,qx/v;q)_{\infty} }{ (cx,dx;q)_{\infty} } \s \frac{ L_n^{(\alpha)}(a,b) \Psi_n^{(\beta)}(y,z)}{(q;q)_n}x^n d_qx.
\end{align}
By the definition of $q$-integral, it can be seen that (\ref{5eq2}) is a linear combination of $L_n^{(\alpha)}(a,b)$. Since $\D_q$ is a difference operator and combined with Proposition \ref{proposition1}, we can easily obtain
\begin{align*}
 \D_{q,a} (1- q^{\alpha} \eta_a) \left\{ I(a,b) \right\}
=-q^{\alpha+1} \eta^2_a  \D_{q,b}  \left\{ I(a,b) \right\}.
\end{align*}
Then by Theorem \ref{thLn}, there exists a sequence $\{\la_n\}$ independent of $a$ and $b$ for which
\begin{align}\label{6eq1}
\int_u^v  \frac{ (qx/u,qx/v;q)_{\infty} }{ (cx,dx;q)_{\infty} } \s
\frac{ L_n^{(\alpha)}(a,b) \Psi_n^{(\beta)}(y,z)}{(q;q)_n}x^n d_qx
=\s \la_n L_n^{(\alpha)}(a,b).
\end{align}
Putting $a=0$ in the above equation, using $L_n^{(\alpha)}(0,b)=b^n$ and (\ref{6gen1}), we find that
\begin{align}\label{6eq6}
I(0,b)&=\int_u^v  \frac{ (qx/u,qx/v;q)_{\infty} }{ (cx,dx;q)_{\infty} } \s
\frac{ \Psi_n^{(\beta)}(y,z)}{(q;q)_n}(bx)^n d_qx \nonumber \\
&=\int_u^v  \frac{ (qx/u,qx/v;q)_{\infty} }{ (cx,dx;q)_{\infty} }
 \frac{(\beta ybx;q)_{\infty} }{ (ybx,zbx;q)_{\infty} }
 d_qx=\s \la_n b^n.
\end{align}
Next, letting $a\to yb$ and $b\to zb$ in Proposition \ref{6pro2} yields
\begin{align*}
 & \int_u^v  \frac{ (qx/u,qx/v,\beta ybx;q)_{\infty} }{ (ybx,zbx,cx,dx;q)_{\infty} } d_qx \\
&=\frac{(1-q)v(q,u/v,qv/u,cduv;q)_{\infty} }
{(cu,cv,du,dv;q)_{\infty}} \s \frac{W_n(c,d,u,v)\Psi_n^{(\beta)}(y,z) }{(q;q)_n} b^n.
\end{align*}
By combining the above $q$-integral with (\ref{6eq6}) and equating the coefficients of $b^n$, we obtain
\begin{align*}
\la_n=\frac{(1-q)v(q,u/v,qv/u,cduv;q)_{\infty}}
{(cu,cv,du,dv;q)_{\infty} } \frac{W_n(c,d,u,v)\Psi_n^{(\beta)}(y,z) }{(q;q)_n}.
\end{align*}
Substituting the above equation into (\ref{6eq1}), Theorem \ref{6th1} follows.
\end{proof}

\begin{remark}\label{6remark}
(1) When $a=b=y=z=0$, Theorem \ref{6th1} immediately reduces to the Proposition \ref{6pro1}, so Theorem \ref{6th1} is really an extension of the Andrews-Askey integral.

(2) When $a=0$ and $b=1$, Theorem \ref{6th1} becomes Proposition \ref{6pro2}.

(3) When $y=0$, $z=1$ and combining (\ref{3Eq1}), we obtain
\begin{align}
& \int_u^v  \frac{ (qx/u,qx/v;q)_{\infty} }{ (bx,cx,dx;q)_{\infty} }
{}_0\phi_1\left( \begin{gathered} - \\ q^{ \alpha+1 } \end{gathered}
;\,q, -q^{\alpha+1} ax \right)  d_qx \nonumber \\
&=\frac{(1-q)v(q,u/v,qv/u,cduv;q)_{\infty}}{(cu,cv,du,dv;q)_{\infty}}
\s \frac{ W_n(c,d,u,v)} { (q;q)_n } L_n^{(\alpha)}(a,b). \label{6eq7}
\end{align}

(4) Setting $d=0$ in (\ref{6eq7}) and noticing that $W_n(c,0,u,v)=\Psi^{(cv)}_n(u,v)$. We immediately obtain following corollary.
\begin{corollary}\label{6Cor1}
For $\max \{ |cu|,|cv|,|bu|,|bv| \}<1$, we have
\begin{align*}
& \int_u^v  \frac{ (qx/u,qx/v;q)_{\infty} }{ (bx,cx;q)_{\infty} }
{}_0\phi_1\left( \begin{gathered} - \\ q^{ \alpha+1 } \end{gathered}
;\,q, -q^{\alpha+1} ax \right)  d_qx \\
&=\frac{(1-q)v(q,u/v,qv/u;q)_{\infty}}{(cu,cv;q)_{\infty}}
\s \frac{ \Psi^{(cv)}_n(u,v) L_n^{(\alpha)}(a,b) } { (q;q)_n } .
\end{align*}
\end{corollary}
\end{remark}

\begin{theorem}\label{6th2}
For $\max \{|au|,|av|,|cu|,|cv|,|du|,|dv|,|\alpha q|,|\beta q| \}<1$, we have
\begin{align*}
& \int_u^v  \frac{ (qx/u,qx/v;q)_{\infty} }{ (cx,dx;q)_{\infty} }
{}_0\phi_1\left(
\begin{gathered} - \\ \alpha q \end{gathered} ;\,q, -\alpha qax \right)
{}_2\phi_1\left(
\begin{gathered} b/a, - \\ \beta q \end{gathered} ;\,q, -ax \right)  d_qx \\
&= \frac{(1-q)v (q,u/v,qv/u,dcuv;q)_{\infty} }{(du,dv,cu,cv;q)_{\infty} }
\s \frac{q^{n(n-1)/2}W_n(d,c,u,v) }{ (\beta q;q)_n (q;q)_n }
 p_n^{(\alpha,\beta)}(a,b).
\end{align*}
\end{theorem}

\begin{proof}
We use $I(a,b)$ to represent the left-hand side of the equation in Theorem \ref{6th2}. It follows from
Corollary \ref{5corollary} that
\begin{align}
& \D_{q,a}  (1-\alpha \eta_a) \left\{ {}_0\phi_1\left(
\begin{gathered} - \\ \alpha q \end{gathered} ;\,q, -\alpha qax \right)
{}_2\phi_1\left( \begin{gathered} b/a, - \\ \beta q \end{gathered} ;\,q, -ax \right) \right\} \nonumber \\
&=-q \D_{q,b}   (\eta^{-1}_b-q \alpha \beta \eta^2_a )
\left\{ {}_0\phi_1\left( \begin{gathered} - \\ \alpha q \end{gathered} ;\,q, -\alpha qax \right) {}_2\phi_1\left(
\begin{gathered} b/a, - \\ \beta q \end{gathered} ;\,q, -ax \right) \right\}.
\end{align}
Then combined with the definition of $I(a,b)$, one can easily obtain
\begin{align*}
\D_{q,a}  (1-\alpha \eta_a) \left\{ I(a,b) \right\}
=-q \D_{q,b}   (\eta^{-1}_b-q \alpha \beta \eta^2_a ) \left\{ I(a,b) \right\}.
\end{align*}
Clearly, $I(a,b)$ is analytic near $(0,0)\in \mathbb{C}^2$. By Theorem \ref{th2}, there exists a sequence $\{\la_n\}$ independent of $a$ and $b$ for which
\begin{align}\label{6eq21}
\int_u^v  \frac{ (qx/u,qx/v;q)_{\infty} }{ (cx,dx;q)_{\infty} }
{}_0\phi_1\left(
\begin{gathered} - \\ \alpha q \end{gathered} ;\,q, -\alpha qax \right)
{}_2\phi_1\left(
\begin{gathered} b/a, - \\ \beta q \end{gathered} ;\,q, -ax \right)  d_qx
=\s \la_n p_n^{(\alpha,\beta)}(a,b).
\end{align}
Putting $a=0$ in (\ref{6eq21}), using $p_n^{(\alpha,\beta)}(0,b)=b^n$, we obtain
\begin{align}\label{6eq4}
I(0,b)&=\int_u^v \s \frac{q^{n(n-1)/2} (bx)^n}{ (\beta q;q)_n (q;q)_n } \frac{ (qx/u,qx/v;q)_{\infty} }{ (cx,dx;q)_{\infty} } d_qx \nonumber \\
&=\s \frac{q^{n(n-1)/2} b^n}{ (\beta q;q)_n (q;q)_n }
\int_u^v \frac{ x^n(qx/u,qx/v;q)_{\infty} }{ (cx,dx;q)_{\infty} } d_qx
=\s \la_n b^n.
\end{align}
We note that interchange the order of summation and the $q$-integral in (\ref{6eq4}) is reasonable, since
\begin{align*}
\s \frac{q^{n(n-1)/2} b^n}{ (\beta q;q)_n (q;q)_n }
\quad \text{and} \quad
\int_u^v \frac{ x^n(qx/u,qx/v;q)_{\infty} }{ (cx,dx;q)_{\infty} } d_qx
\end{align*}
can easily infer that they are converges absolutely and uniformly by using the ratio test. Then by $q$-integral \cite[Proposition 4.3]{L2015}:
\begin{align*}
\int_u^v \frac{ x^n(qx/u,qx/v;q)_{\infty} }{ (bx,cx;q)_{\infty} } d_qx
=\frac{(1-q)v (q,u/v,qv/u,bcuv;q)_{\infty} }{(bu,bv,cu,cv;q)_{\infty} } W_n(b,c,u,v).
\end{align*}
Replacing $b$ by $d$ in above equation yield
\begin{align*}
\int_u^v \frac{ x^n(qx/u,qx/v;q)_{\infty} }{ (dx,cx;q)_{\infty} } d_qx
=\frac{(1-q)v (q,u/v,qv/u,dcuv;q)_{\infty} }{(du,dv,cu,cv;q)_{\infty} } W_n(d,c,u,v).
\end{align*}
Substituting the above equation into (\ref{6eq4}), we have
\begin{align*}
I(0,b)= \s \frac{q^{n(n-1)/2} b^n}{ (\beta q;q)_n (q;q)_n }
\frac{(1-q)v (q,u/v,qv/u,dcuv;q)_{\infty} }{(du,dv,cu,cv;q)_{\infty} } W_n(d,c,u,v)=\s \la_n b^n.
\end{align*}
Equating the coefficients of $b^n$ on both sides of the above equation, we obtain
\begin{align*}
\la_n= \frac{(1-q)v (q,u/v,qv/u,dcuv;q)_{\infty} }{(du,dv,cu,cv;q)_{\infty} }
 \frac{q^{n(n-1)/2}W_n(d,c,u,v) }{ (\beta q;q)_n (q;q)_n }.
\end{align*}
Finally, substituting the above equation into (\ref{6eq21}) and Theorem \ref{6th2} follows.
\end{proof}

\begin{remark}
(1) When $a=b=0$, Theorem \ref{6th2} immediately reduces to the Andrews-Askey integral.

(2) Setting $d=0$ in Theorem \ref{6th2}, we immediately obtain the following corollary.
\begin{corollary}\label{6Cor2}
For $\max \{ |cu|,|cv|,|\alpha q|,|\beta q| \}<1$, we have
\begin{align*}
& \int_u^v  \frac{ (qx/u,qx/v;q)_{\infty} }{ (cx;q)_{\infty} }
{}_0\phi_1\left(
\begin{gathered} - \\ \alpha q \end{gathered} ;\,q, -\alpha qax \right)
{}_2\phi_1\left(
\begin{gathered} b/a, - \\ \beta q \end{gathered} ;\,q, -ax \right)  d_qx \\
&= \frac{(1-q)v (q,u/v,qv/u;q)_{\infty} }{(cu,cv;q)_{\infty} }
\s \frac{q^{n(n-1)/2} }{ (\beta q;q)_n (q;q)_n }
 \Psi^{(cv)}_n(u,v) p_n^{(\alpha,\beta)}(a,b).
\end{align*}
\end{corollary}
\end{remark}

\section{Operator representation of $q$-Laguerre polynomials} \label{Sec5}
We can see that using $q$-partial differential equations is an effective method to study $q$-orthogonal polynomials. Of course, there are also other methods available for studying $q$-orthogonal polynomials. For example, Liu \cite{Liu2022} revealed the essential feature of the Rogers-Szeg\H{o} polynomials by utilizing the method of operators. Naturally, how do operators represent $q$-Laguerre polynomials? For this purpose, we define
\begin{align}\label{wideL1}
L_n^{(\alpha)}(x) \equiv \sum_{k=0}^n  \begin{bmatrix} n \\k \end{bmatrix}
\frac{ q^{k^2+k \alpha} }{ (q^{\alpha+1};q)_k  } (-x)^k=L_n^{(\alpha)}(x,1).
\end{align}
Obviously,
\begin{align}\label{7eq5}
y^n L_n^{(\alpha)}(x/y)=L_n^{(\alpha)}(x,y).
\end{align}
This indicates that $L_n^{(\alpha)}(x)$ and $L_n^{(\alpha)}(x,y)$ can be transformed into each other through variable substitution. The following
Proposition \ref{operator} provides an operator representation of the $q$-Laguerre polynomials $L_n^{(\alpha)}(x)$.
\begin{proposition}\label{operator}
Define operators $\lambda_{\alpha} \{f(\alpha)\}=f(\alpha+1)$ and
\begin{align}\label{operator1}
\Delta_{x,\alpha}= \eta_x-\frac{x q^{\alpha+1}}{1-q^{\alpha+1}} \lambda_{\alpha} \eta_x.
\end{align}
Then we have
\begin{align}\label{wideL}
\Delta_{x,\alpha}^n \{1\}= L_n^{(\alpha)}(x).
\end{align}
\end{proposition}

\begin{proof}
Cigler \cite{Cigler} showed that if $x$ and $y$ are indeterminates such that $xy=qyx$, $q$ commutes with $x$ and $y$, and the associative law holds, then
\begin{align*}
(x+y)^n=\sum_{k=0}^n \begin{bmatrix} n \\k \end{bmatrix} y^k x^{n-k}.
\end{align*}
It follows from the above equation that
\begin{align}\label{Delta}
\Delta_{x,\alpha}^n
=\left( \eta_x-\frac{x q^{\alpha+1}}{1-q^{\alpha+1}} \lambda_{\alpha} \eta_x \right)^n
=\sum_{k=0}^n \begin{bmatrix} n \\k \end{bmatrix} \frac{ q^{k^2+k \alpha} (-x)^k }{(q^{\alpha+1};q)_k}
\lambda_{\alpha}^k \eta_x^n.
\end{align}
Both sides of the above equation acting on the function $f(x) \equiv1$, and then combine with (\ref{wideL1}) to immediately complete the proof.
\end{proof}

We should point out that the operator (\ref{wideL}) is a powerful tool for calculating identities involving $q$-Laguerre polynomials $L_n^{(\alpha)}(x)$. For example, we prove the following bilinear generating function for $q$-Laguerre polynomials.
\begin{theorem}\label{bilinearL}
We have
\begin{align*}
\s \frac{ L^{(\alpha)}_n(x) L^{(\beta)}_n(y)}{(q;q)_n}t^n
&=\frac{1}{(t;q)_{\infty}} \s \sum_{k=0}^{\infty}
\frac{ q^{(n+k)^2+(n+k)\beta+k(k+\alpha)} (xyt)^k (-yt)^n}
{(q;q)_k (q;q)_n (q^{\alpha+1};q)_k (q^{\beta+1};q)_{n+k} } \\
&\quad \times {}_0\phi_1\left(
\begin{gathered} - \\ q^{ \alpha+k+1 } \end{gathered} ;\,q, -xtq^{\alpha+2k+n+1} \right).
\end{align*}
\end{theorem}

\begin{proof}
Appealing to the generating function for $q$-Laguerre polynomials $L_n^{(\alpha)}(x)$ in (\ref{3Eq1}) and then combining it with (\ref{wideL1}), we can obtain
\begin{align*}
\s \frac{ L^{(\beta)}_n(y) } {(q;q)_n} t^n
=\frac{1}{ (t;q)_{\infty} } {}_0\phi_1\left(
\begin{gathered} - \\ q^{ \beta+1 } \end{gathered} ;\,q, -q^{\beta+1} yt \right).
\end{align*}
If we replace $t$ by $t \Delta_{x,\alpha}$ in the above equation, we have
\begin{align}\label{7eq3}
\s \frac{ L^{(\beta)}_n(y) } {(q;q)_n} t^n \Delta_{x,\alpha}^n
= {}_0\phi_1\left( \begin{gathered} - \\ q^{ \beta+1 } \end{gathered} ;\,q, -q^{\beta+1} yt \Delta_{x,\alpha} \right) \frac{1}{ (t \Delta_{x,\alpha};q)_{\infty} }.
\end{align}
Both sides of the above operational equation acting $f(x) \equiv 1$ at same time,
\begin{align}\label{7eq4}
\s \frac{ L^{(\beta)}_n(y) } {(q;q)_n} t^n \Delta_{x,\alpha}^n 1
= {}_0\phi_1\left( \begin{gathered} - \\ q^{ \beta+1 } \end{gathered} ;\,q, -q^{\beta+1} yt \Delta_{x,\alpha} \right) \frac{1}{ (t \Delta_{x,\alpha};q)_{\infty} } 1.
\end{align}
It follows from (\ref{wideL}) that the left-hand side of the above equation becomes
\begin{align}\label{7LHS}
\s \frac{ L^{(\alpha)}_n(x) L^{(\beta)}_n(y)}{(q;q)_n}t^n.
\end{align}
Let RHS denote the right-hand side of the equation (\ref{7eq4}). Calculation shows that
\begin{align*}
{\rm RHS}&={}_0\phi_1\left( \begin{gathered} - \\ q^{ \beta+1 } \end{gathered} ;\,q, -q^{\beta+1} yt \Delta_{x,\alpha} \right) \sum_{m=0}^{\infty} \frac{t^m}{(q;q)_m} \Delta_{x,\alpha}^m \{1\} \\
&=\s \frac{((-1)^n q^{\binom n2})^2 (-q^{\beta+1} yt )^n}{(q;q)_n (q^{ \beta+1 };q)_n} \Delta_{x,\alpha}^n
\left\{ \sum_{m=0}^{\infty} \frac{ L^{(\alpha)}_m(x) } {(q;q)_m} t^m \right\} \\
&=\frac{1}{ (t;q)_{\infty} } \s \frac{((-1)^n q^{\binom n2})^2 (-q^{\beta+1} yt )^n}{(q;q)_n (q^{ \beta+1 };q)_n} \Delta_{x,\alpha}^n \left\{
 {}_0\phi_1\left(
\begin{gathered} - \\ q^{ \alpha+1 } \end{gathered} ;\,q, -q^{\alpha+1} xt \right) \right\}.
\end{align*}
Then use the formula (\ref{Delta}) yields
\begin{align*}
{\rm RHS}&=\frac{1}{ (t;q)_{\infty} } \s \frac{((-1)^n q^{\binom n2})^2 (-q^{\beta+1} yt )^n}{(q;q)_n (q^{ \beta+1 };q)_n} \\
& \quad \times \sum_{k=0}^n \begin{bmatrix} n \\k \end{bmatrix} \frac{ q^{k^2+k \alpha} (-x)^k }{(q^{\alpha+1};q)_k}
\lambda_{\alpha}^k \eta_x^n \left\{
 {}_0\phi_1\left(
\begin{gathered} - \\ q^{ \alpha+1 } \end{gathered} ;\,q, -q^{\alpha+1} xt \right) \right\} \\
&=\frac{1}{ (t;q)_{\infty} } \s \frac{((-1)^n q^{\binom n2})^2 (-q^{\beta+1} yt )^n}{(q;q)_n (q^{ \beta+1 };q)_n} \\
& \quad \times \sum_{k=0}^n \begin{bmatrix} n \\k \end{bmatrix} \frac{ q^{k^2+k \alpha} (-x)^k }{(q^{\alpha+1};q)_k} {}_0\phi_1\left(
\begin{gathered} - \\ q^{ \alpha+k+1 } \end{gathered} ;\,q, -q^{\alpha+k+n+1} xt \right),
\end{align*}
which is equivalent to the right-hand side of the equation in Theorem \ref{bilinearL} by interchanging the order of the summation.
\end{proof}

We find that the combination of the $q$-Laguerre polynomials and other orthogonal polynomials appears on the right-hand side of the $q$-integral (\ref{6eq7}). The following theorem will use the operator method to calculate their generating functions.

\begin{theorem}\label{7th2}
Let $W_n$ defined by (\ref{Wn}), then we have
\begin{align*}
\s \frac{ W_n(c,d,u,v) L_n^{(\alpha)}(x,y) } { (q;q)_n } t^n
&=\frac{1}{(yvt;q)_{\infty} }  \s \sum_{k=0}^{\infty} \frac{q^{k^2+k \alpha} (yut)^n (-xut)^k (cv,dv;q)_{n+k}}{(q,q^{\alpha+1};q)_k (q;q)_n(cduv;q)_{n+k}} \\
& \quad \times {}_0\phi_1\left( \begin{gathered} - \\ q^{ \alpha+k+1 } \end{gathered} ;\,q, -q^{\alpha+n+2k+1} xvt \right).
\end{align*}
\end{theorem}

\begin{proof}
In \cite[(4.10)]{C2014}, Cao proved that
\begin{align*}
\s  \frac{W_n(c,d,u,v) t^n}{(q;q)_n } =\frac{1}{(vt;q)_{\infty} }
{}_2\phi_1\left( \begin{gathered} cv, dv \\ cduv \end{gathered} ;\,q, ut \right).
\end{align*}
Letting $t \to t \Delta_{x,\alpha}$, we obtain
\begin{align*}
\s  \frac{W_n(c,d,u,v) t^n}{(q;q)_n } \Delta_{x,\alpha}^n
={}_2\phi_1\left( \begin{gathered} cv, dv \\ cduv \end{gathered} ;\,q, ut\Delta_{x,\alpha} \right)
\s \frac{(vt)^n}{(q;q)_n } \Delta_{x,\alpha}^n.
\end{align*}
Then acting constant one on both side yields
\begin{align}\label{7eq7}
& \s  \frac{W_n(c,d,u,v) t^n}{(q;q)_n } L_n^{(\alpha)}(x) \nonumber \\
&={}_2\phi_1\left( \begin{gathered} cv, dv \\ cduv \end{gathered} ;\,q, ut\Delta_{x,\alpha} \right)
 \s \frac{(vt)^n}{(q;q)_n } L_n^{(\alpha)}(x) \nonumber \\
&=\frac{1}{(vt;q)_{\infty} }  \s \frac{(cv,dv;q)_n (ut)^n}{(q,cduv;q)_n} \Delta_{x,\alpha}^n
\left\{ {}_0\phi_1\left( \begin{gathered} - \\ q^{ \alpha+1 } \end{gathered} ;\,q, -q^{\alpha+1} xvt \right) \right\}.
\end{align}
Now, using RHS to represent right-hand side of the above equation. Then by (\ref{Delta}), we obtain
\begin{align}\label{7eq8}
{\rm RHS}&=\frac{1}{(vt;q)_{\infty} }  \s \frac{(cv,dv;q)_n (ut)^n}{(q,cduv;q)_n}
\sum_{k=0}^n \begin{bmatrix} n \\k \end{bmatrix} \frac{ q^{k^2+k \alpha} (-x)^k }{(q^{\alpha+1};q)_k} \nonumber \\
& \quad \times \lambda_{\alpha}^k \eta_x^n
\left\{ {}_0\phi_1\left( \begin{gathered} - \\ q^{ \alpha+1 } \end{gathered} ;\,q, -q^{\alpha+1} xvt \right) \right\} \nonumber \\
&=\frac{1}{(vt;q)_{\infty} }  \s \frac{(cv,dv;q)_n (ut)^n}{(cduv;q)_n}
\sum_{k=0}^n \frac{ q^{k^2+k \alpha} (-x)^k }{(q,q^{\alpha+1};q)_k (q;q)_{n-k}} \nonumber \\
& \quad \times  {}_0\phi_1\left( \begin{gathered} - \\ q^{ \alpha+k+1 } \end{gathered} ;\,q, -q^{\alpha+n+k+1} xvt \right) \nonumber \\
&=\frac{1}{(vt;q)_{\infty} }  \s \sum_{k=0}^{\infty} \frac{q^{k^2+k \alpha} (ut)^n (-xut)^k (cv,dv;q)_{n+k}}{(q,q^{\alpha+1};q)_k (q;q)_n(cduv;q)_{n+k}} \nonumber \\
& \quad \times  {}_0\phi_1\left( \begin{gathered} - \\ q^{ \alpha+k+1 } \end{gathered} ;\,q, -q^{\alpha+n+2k+1} xvt \right).
\end{align}
Then by (\ref{7eq7}) and (\ref{7eq8}), we obtain
\begin{align*}
\s \frac{ W_n(c,d,u,v)} { (q;q)_n } L_n^{(\alpha)}(x) t^n
&=\frac{1}{(vt;q)_{\infty} }  \s \sum_{k=0}^{\infty} \frac{q^{k^2+k \alpha} (ut)^n (-xut)^k (cv,dv;q)_{n+k}}{(q,q^{\alpha+1};q)_k (q;q)_n(cduv;q)_{n+k}} \\
& \quad \times {}_0\phi_1\left( \begin{gathered} - \\ q^{ \alpha+k+1 } \end{gathered} ;\,q, -q^{\alpha+n+2k+1} xvt \right).
\end{align*}
By letting $x\to x/y$ and $t\to yt$ in the above equation, and combing it with (\ref{7eq5}), the proof is completed.
\end{proof}

From Theorems \ref{bilinearL} and \ref{7th2}, we see that operator calculation is very convenient, however, we can also calculate by other methods. For that, we introduce the general double basic hypergeometric series is defined as follows \cite[p. 282]{GR}
\begin{align}\label{7eq2}
\Phi^{A:B;C}_{D:E;F} \left[ \begin{gathered}
a_A: b_B; c_C \\ d_D: e_E; f_F \end{gathered} ;\,q; x, y \right]
=\sum_{m=0}^{\infty} \s \frac{(a_A;q)_{m+n} (b_B;q)_m (c_C;q)_n}
{(d_D;q)_{m+n} (q,e_E;q)_m (q,f_F;q)_n} \nonumber \\
\times \left[(-1)^{m+n} q^{\binom {m+n}{2}}\right]^{D-A}
\left[(-1)^m q^{\binom {m}{2}}\right]^{1+E-B}
\left[(-1)^n q^{\binom {n}{2}}\right]^{1+F-C}x^my^n,
\end{align}
where $a_A$ abbreviates the array of $A$ parameters $a_1, a_2, \cdots, a_A$, etc, and $q\neq0$ when $\min \{D-A, 1+E-B, 1+F-C\}<0$. The series (\ref{7eq2}) converges absolutely for $|x|, |y|<1$ when $\min \{D-A, 1+E-B, 1+F-C\}\geq0$ and $|q|<1$. The series (\ref{7eq2}) is called the $q$-Kamp\'{e} de F\'{e}riet series when $B=C$ and $E=F$.
\begin{theorem}\label{7Th1}
If $\max \{|uyt|, |vyt|\}<1$, then, we have
\begin{align*}
& \s \frac{ \Psi_n^{(\beta)}(u,v) L_n^{(\alpha)}(x,y) }{(q;q)_n}t^n \\
& =\frac{ (\beta uyt;q)_{\infty} }{ (uyt,vyt;q)_{\infty} }
\Phi^{0:2;1}_{2:1;0} \left[ \begin{gathered}
-:\beta, vyt; 0 \\ 0, q^{\alpha+1}: \beta uyt; - \end{gathered} ;\,q;
-xutq^{\alpha+1}, -xvtq^{\alpha+1} \right].
\end{align*}
\end{theorem}

\begin{proof}
First of all, applying the $q$-partial derivative operator $\D_{q,t}^k$ to act both sides of the equation (\ref{6gen1}), and then using the formula
(\ref{Dfg1}), we deduce that
\begin{align}\label{7eq1}
\s \frac{\Psi^{(\beta)}_{n+k}(u,v)}{(q;q)_n} t^n
=\frac{ (\beta ut;q)_{\infty} }{(ut,vt;q)_{\infty}}
\sum_{j=0}^k \begin{bmatrix} k \\j \end{bmatrix}
\frac{(\beta, vt;q)_j}{(\beta ut;q)_j} u^j v^{k-j}.
\end{align}
Let LHS to denote the left-hand side of the equation in Theorem \ref{7Th1}, we have
\begin{align*}
{\rm LHS}&=\s \frac{ \Psi^{(\beta)}_n(u,v) } { (q;q)_n }t^n
\sk (-1)^k \begin{bmatrix} n \\k \end{bmatrix}
\frac{  q^{k^2+k\alpha} }{ (q^{\alpha+1};q)_k }x^k y^{n-k} \\
&=\sum_{k=0}^{\infty} \sum_{n=k}^{\infty} \frac{(-1)^k t^n q^{k^2+k\alpha} \Psi^{(\beta)}_n(u,v)}
{(q;q)_k (q;q)_{n-k} (q^{\alpha+1};q)_k  } x^k y^{n-k} \\
&=\sum_{k=0}^{\infty} \frac{(-xt)^k q^{k^2+k\alpha} }
{(q,q^{\alpha+1};q)_k} \s \frac{\Psi^{(\beta)}_{n+k}(u,v)}{(q;q)_n} (yt)^n.
\end{align*}
Letting $t\to yt$ in (\ref{7eq1}), then substituting it into the above equation yields
\begin{align*}
{\rm LHS}&=\sum_{k=0}^{\infty} \frac{(-xt)^k q^{k^2+k\alpha} }
{(q,q^{\alpha+1};q)_k} \frac{ (\beta uyt;q)_{\infty} }{(uyt,vyt;q)_{\infty}}
\sum_{j=0}^k \begin{bmatrix} k \\j \end{bmatrix}
\frac{(\beta, yvt;q)_j}{(\beta yut;q)_j} u^j v^{k-j} \\
&=\frac{ (\beta uyt;q)_{\infty} }{(uyt,vyt;q)_{\infty}}
\sum_{j=0}^{\infty} \sum_{k=0}^{\infty}
\frac{ q^{(k+j)^2+(k+j)\alpha} (\beta, yvt;q)_j (-xtu)^j (-xtv)^k}{ (q^{\alpha+1};q)_{k+j} (q;q)_j (q;q)_k (\beta uyt;q)_j},
\end{align*}
which is equivalent to the right-hand side of the equation in Theorem \ref{7Th1} by series (\ref{7eq2}).
\end{proof}

\begin{remark}
(1) Letting $t\to1$, $x\to a$, $y\to b$ and $\beta\to cv$ in Theorem \ref{7Th1}, and then substituting that into the equation in Corollary \ref{6Cor1}, we obtain
\begin{align*}
& \int_u^v  \frac{ (qx/u,qx/v;q)_{\infty} }{ (bx,cx;q)_{\infty} }
{}_0\phi_1\left( \begin{gathered} - \\ q^{ \alpha+1 } \end{gathered}
;\,q, -q^{\alpha+1} ax \right)  d_qx \\
&=\frac{(1-q)v(q,u/v,qv/u,bcuv;q)_{\infty}}{(bu,bv,cu,cv;q)_{\infty}}
\Phi^{0:2;1}_{2:1;0} \left[ \begin{gathered}
-:cv, bv; 0 \\ 0, q^{\alpha+1}: bcuv; - \end{gathered} ;\,q;
-auq^{\alpha+1}, -avq^{\alpha+1} \right].
\end{align*}

(2) Letting  $\beta=0$ in Theorem \ref{7Th1}, and we immediately obtain the following corollary.
\begin{corollary}\label{7coro1}
If $\max \{|uyt|, |vyt|\}<1$, then, we have
\begin{align*}
 \s \frac{ h_n(u,v) L_n^{(\alpha)}(x,y) }{(q;q)_n}t^n
 =\frac{ 1 }{ (uyt,vyt;q)_{\infty} }
\Phi^{0:1;1}_{2:0;0} \left[ \begin{gathered}
-: vyt ; 0 \\ 0, q^{\alpha+1}:  - ; - \end{gathered} ;\,q;
-xutq^{\alpha+1}, -xvtq^{\alpha+1} \right].
\end{align*}
\end{corollary}
\end{remark}

Applying Corollary \ref{7coro1} to Theorem \ref{RamanujanBeta}, we immediately arrive at the following theorem. The proof will be omitted.
\begin{theorem}\label{7th3}
For $m\in \mathbb{R}$ and $\alpha > -1$, $0<q=e^{-2k^2}<1$ and $\max \{|xzq|,|yzq|\}<1$, we have
\begin{align*}
 & \int_{-\infty}^{+\infty} \frac{e^{-\theta^2+2m\theta}}{(yq^{1/2}e^{2ki \theta};q)_{\infty}  (zq^{1/2}e^{-2ik\theta};q)_{\infty}} \\
 &\quad \times
\Phi^{0:1;1}_{2:0;0} \left[ \begin{gathered}
-: yq^{1/2}e^{2ki \theta} ; 0 \\ 0, q^{\alpha+1}:  - ; - \end{gathered} ;\,q;
x e^{2mki} q^{\alpha+2} , -xv e^{2ki \theta}q^{\alpha+3/2} \right] d \theta \\
&=\sqrt{\pi} e^{m^2} \frac{(-yqe^{2mki};q)_{\infty}  (-zqe^{-2mki};q)_{\infty}}{(yzq;q)_{\infty}}
 {}_0\phi_1\left( \begin{gathered} - \\ q^{ \alpha+1 } \end{gathered} ;\,q, -q^{\alpha+2} xz \right).
\end{align*}
\end{theorem}

\section{Concluding remark}\label{Sec7}

1.  This article interprets homogeneous $q$-Laguerre polynomials and homogeneous little $q$-Jacobi polynomials mainly from the perspective of $q$-partial differential equations, providing a new method for studying these two $q$-orthogonal polynomials. This research method also belongs to Liu's theory of
$q$-partial differential equations. Meanwhile, we easily found the bilinear generating function of the $q$-Laguerre polynomials by using the newly constructed operator.

2.  The following $q$-integral formula \cite[Proposition 13.8]{L2013} with six parameters is a generalization of Andrews-Askey integral formula. It will be used later in the discussion.
\begin{proposition}\label{6proposition}
If $a,b,c,d,u,v,r$ are complex numbers such that $\max\{|au|,|bu|,|cu|,|av|,|bv|,\\|cv|,|abr/c|\}<1$ and $uv\neq0$, then we have the following $q$-integral formula:
\begin{align*}
\int_u^v \frac{(qx/u,qx/v,abrx;q)_{\infty} }{(ax,bx,cx;q)_{\infty} }d_qx
&=\frac{(1-q)v(q,u/v,qv/u;q)_{\infty} (acuv,bcuv,abr/c;q)_{\infty} }
{(au,av,bu,bv,cu,cv;q)_{\infty}} \\
& \quad \times {}_3\phi_2\left( \begin{gathered}
cu,cv,cuv/r \\ acuv, bcuv \end{gathered} ;\,q, \frac{abr}{c} \right).
\end{align*}
\end{proposition}
It follows from Proposition \ref{proposition1} and (\ref{Eq5}) that
\begin{align*}
&\D_{q,a} (1- q^{\alpha} \eta_a) \left\{ \frac{ (\gamma bx;q)_{\infty} }{ (bx;q)_{\infty} } {}_1\phi_2\left(
\begin{gathered} \gamma \\ q^{ \alpha+1 }, \gamma bx \end{gathered}
;\,q, -q^{\alpha+1} ax \right) \right\} \\
&=-q^{\alpha+1} \eta^2_a  \D_{q,b}  \left\{ \frac{ (\gamma bx;q)_{\infty} }{ (bx;q)_{\infty} } {}_1\phi_2\left(
\begin{gathered} \gamma \\ q^{ \alpha+1 }, \gamma bx \end{gathered}
;\,q, -q^{\alpha+1} ax \right) \right\}.
\end{align*}
For arbitrarily given $\gamma$, define the $q$-integral as follows
\begin{align*}
\varphi(a,b)\equiv \int_u^v  \frac{ (\gamma bx,qx/u,qx/v;q)_{\infty} }{ (bx,cx,dx;q)_{\infty} } {}_1\phi_2\left(
\begin{gathered} \gamma \\ q^{ \alpha+1 }, \gamma bx \end{gathered}
;\,q, -q^{\alpha+1} ax \right) d_qx.
\end{align*}
Clearly, letting $\gamma=0$ in function $\varphi(a,b)$, which degenerates to the left-hand side of the equation in the Remark \ref{6remark}(3). It is easy to show that $\varphi(a,b)$ is analytic near $(0,0)\in \mathbb{C}^2$ and satisfies the equation
\begin{align*}
 \D_{q,a} (1- q^{\alpha} \eta_a) \left\{ \varphi(a,b) \right\}
=-q^{\alpha+1} \eta^2_a  \D_{q,b}  \left\{ \varphi(a,b) \right\}.
\end{align*}
Then by Theorem \ref{thLn}, there exists a sequence $\{\Lambda_n\}$ independent of $a$ and $b$ for which
\begin{align}\label{6eq2}
\varphi(a,b)=\s \Lambda_n L_n^{(\alpha)}(a,b).
\end{align}
Putting $a=0$ in the above equation, using $L_n^{(\alpha)}(0,b)=b^n$, we find that
\begin{align}\label{6eq5}
\varphi(0,b)=\int_u^v  \frac{ (\gamma bx,qx/u,qx/v;q)_{\infty} }{ (bx,cx,dx;q)_{\infty} } d_qx =\s \Lambda_n b^n.
\end{align}
Then setting $a\to d$ and $r\to \gamma/d$ in Proposition \ref{6proposition}, then, we have
\begin{align*}
\int_u^v \frac{(qx/u,qx/v,\gamma bx;q)_{\infty} }{(dx,bx,cx;q)_{\infty} }d_qx
&=\frac{(1-q)v(q,u/v,qv/u;q)_{\infty} (dcuv,bcuv,\gamma b/c;q)_{\infty} }
{(du,dv,bu,bv,cu,cv;q)_{\infty}} \\
& \quad \times {}_3\phi_2\left( \begin{gathered}
cu,cv,cduv/\gamma \\ dcuv, bcuv \end{gathered} ;\,q, \frac{\gamma b}{c} \right).
\end{align*}
Substituting the above equation into (\ref{6eq5}), we find that sequence $\{\Lambda_n\}$ is determined by the following equation
\begin{align}\label{6eq3}
&\frac{(1-q)v(q,u/v,qv/u;q)_{\infty} (dcuv,bcuv,\gamma b/c;q)_{\infty} }
{(du,dv,bu,bv,cu,cv;q)_{\infty}} \nonumber \\
& \times {}_3\phi_2\left( \begin{gathered}
cu,cv,cduv/\gamma \\ dcuv, bcuv \end{gathered} ;\,q, \frac{\gamma b}{c} \right)=\s \Lambda_n b^n.
\end{align}
If we can obtain the expression of $\Lambda_n$ and then substituting it into (\ref{6eq2}), we can calculate the value of $q$-integral $\varphi(a,b)$. However, it seems that this is not an easy thing, so we left as
an open problem for the readers.

\section*{Acknowledgments}

The authors sincerely thank Associate Professor Min-Jie Luo from the Department of Mathematics and Statistics at Donghua University, as well as Professor Jian Cao from the School of Mathematics at Hangzhou Normal University for their valuable suggestions.

\section*{Conflict of interest}
The authors declare that they have no conflict of interest.

\end{document}